\documentclass[10pt, reqno]{amsart}

\numberwithin{equation}{section}
\usepackage{amssymb, color}
\usepackage{hyperref}

\newcommand{\qtq}[1]{\quad\text{#1}\quad}

\let\Re=\undefined\DeclareMathOperator*{\Re}{Re}
\let\Im=\undefined\DeclareMathOperator*{\Im}{Im}

\newcommand{\R}{\mathbb{R}}
\newcommand{\C}{\mathbb{C}}

\newcommand{\eps}{\varepsilon}

\newcommand{\Real}{\mathbb R}

\newtheorem{theorem}{Theorem}[section]

\newtheorem{lemma}[theorem]{Lemma}
\newtheorem{corollary}[theorem]{Corollary}

\newtheorem{proposition}[theorem]{Proposition}

\theoremstyle{definition}

\newtheorem{remark}[theorem]{Remark}

\theoremstyle{remark}

\usepackage{mathtools}
\mathtoolsset{showonlyrefs}
\usepackage[abbrev]{amsrefs}

\begin{document}

\title[Threshold solutions]{Threshold solutions for the intercritical inhomogeneous NLS} 

\author{Luccas Campos}
\author{Jason Murphy}

\begin{abstract}
We consider the focusing inhomogeneous nonlinear Schrödinger equation in $H^1(\mathbb{R}^3)$,
\begin{equation}
i\partial_t u + \Delta u + |x|^{-b}|u|^{2}u=0,\\
\end{equation}
where $0 < b <\tfrac{1}{2}$. Previous works (see e.g. \cites{Campos21radial, MurphyINLS, FG17}) have established a blowup/scattering dichotomy below a mass-energy threshold determined by the ground state solution $Q$.  

In this work, we study solutions exactly at this mass-energy threshold.  In addition to the ground state solution, we prove the existence of solutions $Q^\pm$, which approach the standing wave in the positive time direction, but either blow up or scatter in the negative time direction. Using these particular solutions, we classify all possible behaviors for threshold solutions.  In particular, the solution either behaves as in the sub-threshold case, or it agrees with $e^{it}Q$, $Q^+$, or $Q^-$ up to the symmetries of the equation.
\end{abstract}

\maketitle


\section{Introduction}

We consider the initial-value problem for inhomogeneous nonlinear Schr\"odinger equations (NLS) of the form
\begin{equation}\label{NLS}
\begin{cases} i\partial_t u + \Delta u + |x|^{-b}|u|^2 u = 0, \\
u|_{t=0}=u_0\in H^1(\R^3),
\end{cases}
\end{equation}
where $u:\R_t\times\R_x^3\to\C$ and $b\in(0,\frac12)$.  We define $s_c=\tfrac{1+b}{2}\in(\tfrac12,\tfrac34)$, so that $\dot H^{s_c}(\R^3)$ is the critical Sobolev space of initial data for \eqref{NLS}. 

This model arises in the setting of nonlinear optics, where the factor $|x|^{-b}$ represents some inhomogeneity in the medium (see e.g. \cite{Gill, Liu}).  As pointed out by Genoud and Stuart \cite{g_8}, the factor $|x|^{-b}$ appears naturally as a limiting case of potentials that decay polynomially at infinity.

We denote the ground state for \eqref{NLS} by $Q$.  That is, $Q$ is the unique nonnegative, radial solution to 
\[
-Q+\Delta Q + |x|^{-b}Q^3 = 0,
\]
and $u(t)=e^{it}Q$ is a global, non-scattering solution to \eqref{NLS} (the \textit{ground state solution}).  Several works (see e.g. \cite{FG17, FG20, MMZ19, CFGM, MurphyINLS, Campos21radial, CC22}) have considered the behavior of solutions below the mass-energy threshold determined by $Q$, i.e. for solutions satisfying
\[
M(u)^{1-s_c}E(u)^{s_c}<M(Q)^{1-s_c}E(Q)^{s_c}.
\]

We call such solutions \emph{sub-threshold} solutions.  In particular, in this regime one has a scattering/blowup dichotomy given in terms of the size of the mass and kinetic energy, namely:
\[
\begin{cases}
\|u_0\|_{L^2}^{1-s_c}\|\nabla u_0\|_{L^2}^{s_c} < \|Q\|_{L^2}^{1-s_c} \|\nabla Q\|_{L^2}^{s_c} \implies \text{scattering}, \\
\|u_0\|_{L^2}^{1-s_c}\|\nabla u_0\|_{L^2}^{s_c} > \|Q\|_{L^2}^{1-s_c} \|\nabla Q\|_{L^2}^{s_c} \implies \text{blowup}.
\end{cases}
\]
Here \emph{scattering} (as $t\to\pm\infty$) refers to the fact that there exist $u_\pm\in H^1$ such that
\[
\lim_{t\to\pm\infty}\|u(t)-e^{it\Delta}u_\pm\|_{H^1}=0. 
\]

Some recent work has also considered the long-time behavior of solutions beyond the ground state threshold.  In particular, in \cite{CCnodea21}, the first author and Cardoso established a dichotomy for fast-decaying initial data satisfying
\begin{equation}
    M(u_0)^{1-s_c}E(u_0)^{s_c}\left(1-\frac{(V'(0))^2}{32 E(u_0) V(0)} \right) \leq 1,
\end{equation}
where $V(t) = \int |x|^2 |u(t)|^2 \, dx$. 
This result classifies the long-time behavior for a (non-empty) set of initial data with arbitrarily large mass and energy. Unlike the sub-threshold case, the results obtained are generally not symmetric in time, as the classification depends on the sign of $V'(0)$, which is changed after applying the time-reversal symmetry.

In this paper, we study the behavior of solutions to \eqref{NLS} with data satisfying
\[
M(u_0)^{1-s_c}E(u_0)^{s_c} = M(Q)^{1-s_c}E(Q)^{s_c},
\]
which we call \emph{threshold solutions}. Using the scaling symmetry, it is equivalent to study initial data satisfying
\[
M(u_0)=M(Q) \qtq{and} E(u_0)=E(Q). 
\]

For such solutions, we further consider whether
\[
\|\nabla u_0\|_{L^2} < \|\nabla Q\|_{L^2} \qtq{or} \|\nabla u_0\|_{L^2} > \|\nabla Q\|_{L^2}.
\]

The variational characterization of $Q$ then implies 
\begin{equation}
\|\nabla u(t)\|_{L^2} < \|\nabla Q\|_{L^2} \qtq{or} \|\nabla u(t)\|_{L^2} > \|\nabla Q\|_{L^2},
\end{equation}
respectively, for all $t$ in the lifespan of $u$. We call the corresponding solutions \emph{constrained} or \emph{unconstrained}, respectively.  We also note that if $
\|\nabla u_0\|_{L^2}=\|\nabla Q\|_{L^2}$, then $u\equiv e^{it}Q$ modulo the symmetries of the equation. 

The classification of threshold behaviors has been a topic of recent mathematical interest.  In the setting of pure-power NLS, the first such result was due to Duyckaerts and Merle \cites{DM_Dyn} for the energy-critical problem (see also \cite{higher_thre} for the higher-dimensional case).  Similar work has also appeared in the setting of the energy-critical wave equation (see e.g. \cite{DM_Dyn_wave}). In the intercritical setting, Duyckaerts and Roudenko \cite{DR_Thre} addressed the (homogeneous) cubic NLS in three dimensions; this was later generalized to the full intercritical range in any dimension by the first author, together with Farah and Roudenko \cite{campos2022threshold}.  Apart from the pure power-type NLS, we would also like to point out the work of Yang, Zeng and Zhang \cite{YZZ_thre}, who considered the energy-critical NLS in the presence of an inverse-square potential.  Finally, we would like to mention some related works on the phenomenon of threshold scattering (see e.g. \cite{DLR22, MMZ22, KMV21}).

Our main results revolve around the existence and uniqueness of certain orbits. First, we prove the existence of two particular solutions to \eqref{NLS}.

\begin{theorem}\label{T:Existence_special_solutions} There exist radial solutions $Q^\pm$ to \eqref{NLS} with
\[
M(Q^\pm)=M(Q)\qtq{and} E(Q^\pm)=E(Q),
\]
defined on intervals $I^\pm\supset[0,\infty)$, which satisfy
\begin{equation}
\|Q^\pm(t)-e^{it}Q\|_{H^1} \lesssim e^{-ct} 
\end{equation}
for some $c>0$ and all $t>0$.

The solution $Q^-$ is global ($I^-=\R$), satisfies
\[
\|\nabla Q^-(0)\|_{L^2}<\|\nabla Q\|_{L^2},
\]
and scatters in $H^1(\mathbb R^3)$ as $t\to-\infty$.

The solution $Q^+$ satisfies
\[
\|\nabla Q^+(0)\|_{L^2}>\|\nabla Q\|_{L^2},
\]
and blows up in finite negative time ($I^+=(T_-,\infty)$ for some $T_-<0$). Moreover, $xQ^+ \in L^2(\mathbb{R}^3)$.

\end{theorem}

Using the solutions obtained in Theorem~\ref{T:Existence_special_solutions}, we can classify all threshold solutions to \eqref{NLS}.

\begin{theorem}[Classification of threshold dynamics]\label{T:Classification_threshold_solutions} If $u_0\in H^1$ satisfies
\[
M(u_0)^{1-s_c}E(u_0)^{s_c}=M(Q)^{1-s_c}E(Q)^{s_c},
\]
then we have the following:
\begin{itemize}
\item[i)] If 
\[
\|u_0\|_{L^2}^{1-s_c}\|\nabla u_0\|_{L^2}^{s_c} < \|Q\|_{L^2}^{1-s_c}\|\nabla Q\|_{L^2}^{s_c},
\]
then $u$ either scatters as $t\to\pm\infty$ or $u = Q^-$ up to symmetries.
\item[ii)] If 
\[
\|u_0\|_{L^2}^{1-s_c}\|\nabla u_0\|_{L^2}^{s_c} = \|Q\|_{L^2}^{1-s_c}\|\nabla Q\|_{L^2}^{s_c},
\]
then $u=e^{it}Q$ up to symmetries.
\item[iii)] If
\[
\|u_0\|_{L^2}^{1-s_c}\|\nabla u_0\|_{L^2}^{s_c} > \|Q\|_{L^2}^{1-s_c}\|\nabla Q\|_{L^2}^{s_c}
\]
and $u_0$ is radial or $x u_0 \in L^2(\mathbb R^3)$, then $u$ either blows up in finite positive and negative times or $u = Q^+$ up to symmetries.
\end{itemize}
\end{theorem}
\begin{remark} The assertion that $u=v$ \emph{up to symmetries of \eqref{NLS}} means that there exist $\lambda_0 >0$, $\theta_0 \in \mathbb{R}/{2\pi\mathbb{Z}}$, and $t_0 \in \mathbb{R}$ such that either
\begin{equation}
u(t,x) = e^{i\theta_0}\lambda_0^{\frac{2-b}{2}}v(\lambda_0^2t-t_0,\lambda_0 x)
\qtq{or}u(t,x) = e^{i\theta_0}\lambda_0^{\frac{2-b}{2}}\overline{v}(\lambda_0^2t-t_0,\lambda_0x).
\end{equation}
That is, $u$ and $v$ agree up to scaling, phase, time-translation and time-reversal.  Note that all these symmetries leave the $\dot H_x^{s_c}$-norm invariant. 
\end{remark}

 \begin{remark} All cases in Theorem \ref{T:Classification_threshold_solutions} do occur. Indeed, by Theorem~\ref{T:Existence_special_solutions}, one only needs to check that blowup in finite positive and negative times and scattering in both time directions are possible. In fact, this follows from the dichotomy proved in \cite{CCnodea21} (which gives results that are symmetric in time for real initial data, since then $V'(0) = 0$).  %

\end{remark}

The argument for Theorems~\ref{T:Existence_special_solutions}~and~\ref{T:Classification_threshold_solutions} proceeds as follows:  

The first main step is to prove that in certain scenarios, forward-global threshold solutions necessarily converge to the ground state solution (with an exponential rate).  We show this first in the setting of a constrained solution that fails to scatter (see Section~\ref{Sec:constrained}).  The idea is first to establish some compactness properties for such solution (which is achieved via concentration compactness and the sub-threshold dichotomy results), and then use a combination of virial estimates and so-called modulation analysis to establish the desired convergence property.  We next prove convergence in the setting of an unconstrained solution, relying once again on virial estimates and modulation analysis (see Section~\ref{Sec:unconstrained}).  Modulation analysis, which refers to obtaining a suitable decomposition of the solution during times when it approaches the orbit of $Q$, is prerequisite to both of these arguments; accordingly, we carry out this analysis earlier in the paper, in Section~\ref{Sec:modulation}.  This analysis relies in turn on a spectral analysis of the operator $\mathcal{L}$ arising in the linearization of \eqref{NLS} around the ground state solution, which we carry out in Section~\ref{Sec:spectral}. 

The second main step (carried out in Section~\ref{Sec:special}) is to establish the existence of solutions behaving in the manner described above.  That is, we prove the existence of forward-global solutions converging exponentially to the ground state.  For this part of the argument, we first use explicit functions related to the spectrum of $\mathcal{L}$ (obtained in Section~\ref{Sec:spectral}) to build good approximate solutions, and then utilize a fixed point argument to obtain true solutions.  The solutions we build are essentially the particular solutions $Q^\pm$ appearing in Theorem~\ref{T:Existence_special_solutions}.

Finally, the third step (carried out in Section~\ref{Sec:bootstrap}) is to establish a uniqueness property for solutions converging exponentially to the ground state.  Combining the first and second steps above, we can then obtain the rather rigid statements appearing in Theorem~\ref{T:Classification_threshold_solutions}, namely, that nonscattering constrained solutions must coincide with $Q^-$, while forward-global unconstrained solutions must coincide with $Q^+$.

In Section~\ref{Sec:closure}, we put together all of the pieces and quickly complete the proof of the main theorems. 

The inhomogeneity $|x|^{-b}$ brings some new challenges compared compared to the homogeneous case; in particular, it introduces a singularity at the origin, which gets stronger after being differentiated. In a few instances throughout the paper, we have to work with a restricted range of $b$ precisely because of this issue. In particular, the modulation analysis of Section~\ref{Sec:modulation} leads us to the restriction $b\in(0,\tfrac12)$.  In addition, the inhomogeneity breaks the translation symmetry (thus breaking conservation of momentum and Galilean invariance).  As translation parameters appear in the profile decomposition adapted to the linear evolution $e^{it\Delta}$, we have to be careful when passing to the nonlinear profile decomposition in the constrained case; in particular, it is essential to show that these translation parameters may always be chosen to be identically zero.  This is ultimately possible due to the fact that in the regime $|x|\to\infty$, the equation \eqref{NLS} is well-approximated by the underlying \emph{linear} equation, which guarantees that profiles with diverging translation parameters always correspond to scattering solutions (and hence do not appear when we consider a non-scattering threshold solution).  Finally, the singularity must be treated carefully as as we work to establish decay and regularity of the ground state and other functions related to the spectrum of the linearized operator; these properties play an important role in the construction of the special solutions in Section~\ref{Sec:special}.

\textbf{Acknowledgments}. L. C. was financed by grant \#2020/10185-1, São Paulo Research Foundation (FAPESP). J. M. was supported by a Simons Collaboration Grant. 

\section{Preliminaries}\label{Sec:prelim}

We write $A\lesssim B$ to denote $A\leq CB$ for some $C>0$.  If $A\lesssim B$ and $B\lesssim A$ then we write $A\sim B$.  We also make use of the standard `big-oh' notation, $\mathcal{O}$.  We write $(\cdot,\cdot)$ for the standard $L^2$ inner product.

\subsection{The ground state}  The ground state $Q$ is the unique nonnegative, radial, $H^1$-solution to
\begin{equation}\label{Q-PDE}
-Q+\Delta Q + |x|^{-b}Q^3 = 0.
\end{equation}

It may be constructed as an optimizer to the following sharp Gagliardo--Nirenberg inequality:
\begin{equation}\label{sharpGN}
\| |x|^{-b} f^4 \|_{L^1} \leq C_{GN} \|f\|_{L^2}^{1-b}\|\nabla f\|_{L^2}^{3+b}.
\end{equation}

The \textit{Pohozaev} identities for $Q$ are obtained by multiplying the \eqref{Q-PDE} by $Q$ or $x\cdot\nabla Q$ and integrating by parts (see \cite{farah} for more details).  They read as follows:
\begin{align}
-\|Q\|_{L^2}^2 - \|\nabla Q\|_{L^2}^2 + \| |x|^{-b} Q^4\|_{L^1} & = 0, \\
 \tfrac32\|Q\|_{L^2}^2 + \tfrac12\|\nabla Q\|_{L^2}^2 - \tfrac{(3-b)}{4}\| |x|^{-b} Q^4\|_{L^1} &= 0.
\end{align}
Combining these identities, we may derive that 
\begin{equation}\label{energy-of-Q}
\|\nabla Q\|_{L^2}^2 = \tfrac{3+b}{4}\||x|^{-b}Q^4\|_{L^1}, \qtq{so that} E(Q) = [\tfrac12-\tfrac1{3+b}]\|\nabla Q\|_{L^2}^2.
\end{equation}

\subsection{Well-posedness and stability}\label{S:WPS}

We define the scattering norm 
\[
\|u\|_{S(I)}=\|u\|_{L_t^4 L_x^{\frac{6}{1-b}}(I\times\R^3)}, 
\]
where $L_t^q L_x^r$ denotes the standard mixed Lebesgue norm. We further define the Strichartz norm
\begin{equation}
    \|u\|_{Z(I)} = \sup_{(q,r)\in A}\|u\|_{L^q_t W^{1,r}_x(I\times \mathbb R^3)},
\end{equation}
where $A$ is the set of \textit{admissible pairs}
\begin{equation}
A = \bigl\{(q,r) \colon \tfrac{2}{q}+\tfrac{3}{r} = \tfrac{3}{2}, \,\, q,r \geq 2\bigr\}.
\end{equation}

By Sobolev embedding, we have
\begin{equation}
\|u\|_{S(I)}\lesssim \||\nabla|^{s_c} u\|_{L^{4}_tL^{3}_x(I\times \mathbb R^3)} \lesssim \|u\|_{Z(I)}.
\end{equation}

We also define the dual norm

\begin{equation}
    \|u\|_{N(I)} = \|u\|_{L^2_tW^{1,6/5}_x(I\times \mathbb R^3)}.
\end{equation}

The relationship between the norms $Z(I)$ and $N(I)$ is given by the so-called \textit{Strichartz estimates}:

\begin{lemma}[Strichartz estimates] If $e^{it\Delta}$ is the evolution associated to the linear equation
\begin{equation}
    i\partial_t u + \Delta u = 0,
\end{equation}
then
\begin{align}
    \left\|e^{it\Delta}f\right\|_{Z(\mathbb R)}&\lesssim\left\|f\right\|_{H^1_x},\\
    \left\|\int_{I} e^{i(t-s)\Delta}F(s)\, ds\right\|_{Z(I)}&\lesssim\left\|F\right\|_{N(I)}.
\end{align}
\end{lemma}

The presence of the inhomogeneous factor $|x|^{-b}$ in \eqref{NLS} suggests we employ either Sobolev or Hardy-type inequalities while estimating the nonlinear term. Those can be combined in a unified way, yielding
\begin{equation}\label{stein_lemma}
\left\| |x|^{-\beta}u\right\|_{L^{q}_x} \lesssim
\left\||\nabla|^su\right\|_{L^p_x}.
\end{equation}
provided  $1 < p \leq q <\infty$, $0 < s < 3$ and $\beta \geq 0$ satisfy
$$
\beta < \tfrac{3}{q}, \quad \quad s= \tfrac{3}{p}-\tfrac{3}{q}+\beta
$$
(see \cite[Theorem B$^*$]{stein}).

The local well-posedness for \eqref{NLS} was first studied by Genoud and Stuart in \cite{g_8} (see also Genoud \cite{g_6}) by an approach using energy estimates as in Cazenave \cite{cazenave} (i.e., without relying on Strichartz inequalities).  They established a well-posedness result in $C^{0}_t H^1_x(I\times\mathbb R^3) \cap C^1_t H^{-1}_x(I\times\mathbb R^3)$ for the range $0 < b < 2$.  More recently, Guzmán \cite{Guzman_Well17}, Dinh \cite{Dinh_Well17} and the first author \cite{Campos21radial} proved that if $0 < b < 3/2$, the solutions also belong locally (in time) to $Z(I)$.

We also have the following stability result for \eqref{NLS}:
\begin{proposition}[Stability, c.f. \cite{FG17}]
Let $v$ be a solution to 

\begin{equation}
    i \partial_t v+\Delta v  + |x|^{-b}|v|^2v = e
\end{equation}
which satisfies
\begin{equation}
    \|v\|_{L^\infty_tH^1_x(I\times \Real^3)} + \|v\|_{S(I)} \leq M < +\infty.
\end{equation}

Then there exists $\epsilon_1 = \epsilon_1(M)>0$ such that if $u_0 \in H^1$ satisfies $\|u_0-v(0)\|_{H^1}<\epsilon$ and
\begin{equation}
    { \|e\|_{L^{\frac{8}{3(1-b)}}_t L^{\frac{12}{9+b}}_x(I\times \Real^3)}+}
    \|e\|_{N(I)} < \epsilon
\end{equation}

for some $0<\epsilon<\epsilon_1$, then there exists a unique solution $u$ to \eqref{NLS} on $I$ with $u(0) = u_0$ and
\begin{equation}
    \|u-v\|_{S(I)} \lesssim_M \epsilon.
\end{equation}

\end{proposition}

For constrained solutions, we have the following:
\begin{proposition}[Global well-posedness, \cite{farah}]\label{P:GWP} Suppose
\[
M(u_0)^{1-s_c}E(u_0)^{s_c} \leq M(Q)^{1-s_c}E(Q)^{s_c}
\]
and
\[
\|u_0\|_{L^2}^{1-s_c}\|\nabla u_0\|_{L^2}^{s_c}\leq\|Q\|_{L^2}^{1-s_c}\|\nabla Q\|_{L^2}^{s_c}.
\]
Then the corresponding solution to \eqref{NLS} is global in time and remains uniformly bounded in $H^1_x$. 
\end{proposition}

\section{Spectral properties of the linearized operator}\label{Sec:spectral}

We will consistently make use of properties of solutions to the \textit{linearized equation} around the ground state.  In particular, if $u$ solves \eqref{NLS}, writing $u = e^{it}(Q+v)$, one has
\begin{equation}\label{linearized_1}
    i\partial_t v + \Delta v + K(v) + R(v) = 0,
\end{equation}
where
\begin{equation}\label{def:KR}
K(v) = |x|^{-b}Q^2(3v_1+iv_2)\qtq{and} R(v) = |x|^{-b}Q^3 G(Q^{-1} v), 
\end{equation}
with 
\begin{align*}
G(z) & = |1+z|^2(1+z)-1- 2z - \overline{z} \\
&= 3z_1^2+z_2^2+z_1^3+z_1z_2^2+i(2z_1z_2 + z_1^2z_2 + z_2^3).
\end{align*}
Note that $G(0) = G_z(0) = G_{\overline{z}}(0) = 0$.

By identifying $a+bi\in\mathbb{C}$ with $[a,b]^t\in\R^2$, equation \eqref{linearized_1} can be rewritten as 
\begin{equation}\label{linearized_eq}
    \partial_t v + \mathcal{L}v = iR(v),
\end{equation}
where
\[
\mathcal{L} = \begin{pmatrix} 0 & L_-\\-L_+ & 0\end{pmatrix},
\]
with
\[
L_+ = 1 -\Delta - 3|x|^{-b}Q^2,\quad L_- = 1-\Delta - |x|^{-b}Q^2.
\]
We call \eqref{linearized_eq} the \textit{linearized equation}, and $\mathcal{L}$ the \textit{linearized operator}. 

It is immediate to check that $L_- Q = 0$. Moreover, \[
L_+Q = -\Delta Q - 3 |x|^{-b}Q^3 = -2|x|^{-b}Q^3,
\]
so that $(L_+ Q, Q)_{L^2} < 0$. 

We first show that there can be only two non-negative directions for $\mathcal{L}$.
\begin{proposition}\label{P:coercive}
For all $v \in {H}^1(\mathbb{R}^3,\mathbb{R})$,
\begin{itemize}
\item $(L_-v,v)\gtrsim \|v\|_{{H}^1}^2$, if $( v,  Q) =  0$, 
\item $(L_+v,v)\gtrsim \|v\|_{{H}^1}^2$, if $( v,  \Delta Q) = 0$.
\end{itemize}
\end{proposition}
\begin{proof}

\textbf{Step 1:} Writing $\langle\nabla\rangle^{-1} L_\pm\langle\nabla\rangle^{-1} = I - K_{\pm}$, we claim that $K_{\pm}: L^2 \to L^2$ is compact. To see this, we will prove that $ |x|^{-b}Q^2\langle\nabla\rangle^{-1}: L^2 \to L^2$ is compact.  Indeed, for $v\in L^2$, we can obtain the following estimates:

First, we have
\begin{equation}
\||x|^{-b}Q^2\langle\nabla\rangle^{-1}v\|_{L^2} \lesssim \||x|^{-b}Q^2\|_{L^3}\|\langle\nabla\rangle^{-1}v\|_{L^6} \lesssim \|v\|_{L^2},
\end{equation}

Next, 
\begin{align}
\|&|\nabla|^{\frac{1-b}{2}}(|x|^{-b}Q^2\langle\nabla\rangle^{-1}v)\|_{L^2} 
\\ &\lesssim 
\|\nabla(|x|^{-b}Q^2\langle\nabla\rangle^{-1}v)\|_{L^{\frac{6}{b+4}}} \\&\lesssim \||x|^{-b-1}Q^2\|_{L^{\frac{6}{b+3}}}\|\langle\nabla\rangle^{-1}v\|_{L^6} + \||x|^{-b}Q|\nabla Q|\|_{L^{\frac{6}{b+1}}}\|\langle\nabla\rangle^{-1}v\|_{L^2} \\
& \quad + \||x|^{-b}Q^2\|_{L^{\frac{6}{b+1}}}\|\nabla\langle\nabla\rangle^{-1}v\|_{L^2}\\
&\lesssim  \|v\|_{L^2}.
\end{align}
Finally, 
\begin{equation}
\||x|^{-b}Q^2\langle\nabla\rangle^{-1}v\|_{L^2(\{|x|\geq R\})} \lesssim \tfrac{1}{R^b} \|Q\|^2_{L^{\infty}}\|\langle\nabla\rangle^{-1}v\|^2_{L^{2}} \lesssim  \tfrac{1}{R^b}\|v\|_{L^2}.
\end{equation}
The desired compactness then follows from Rellich--Kondrachov.

We thus have that the eigenvalues of $I-K_{\pm}$ are discrete, and can only accumulate at $1$. Therefore, in the interval $(-\infty,\tfrac{1}{2}]$, say, $I-K_{\pm}$ has at most a finite number $\{\lambda_i^{\pm}\}_{i=1}^{N_{\pm}}$ of eigenvalues (counting multiplicity), which we assume are ordered in a non-decreasing way. By Weyl's Theorem, the essential spectrum is $[1,+\infty)$.

\textbf{Step 2:} Since $Q$ is the minimizer of the Weinstein functional
\begin{equation}
    J(f) = \frac{
    \left(\displaystyle\int |\nabla f|^2\right)^{\frac{b+3}{2}}
    \left(\displaystyle\int | f|^2\right)^{\frac{1-b}{2}}
    }{\displaystyle\int |x|^{-b}|f|^4},
\end{equation}
by writing $f = v_1 + i v_2$, the conditions $\frac{d^2}{d\epsilon^2}J(Q+\epsilon v)|_{\epsilon=0}\leq 0$ and $( v_1,  \Delta Q) = 0$ give
\begin{equation}
(L_+v_1,v_1)+(L_-v_2,v_2) \gtrsim [1+3b+ b(1-b)]\left(\int Q v_1\right)^2.
\end{equation}
Therefore, since $0<b<1$, we deduce that $L_+$ is non-negative if $(v_1, \Delta Q) = 0$, and that $L_-$ is always non-negative.

\textbf{Step 3:} We already showed that $\lambda_1^+<0$, $\lambda_2^+ \geq 0$ and $\lambda_1^-=0$. To finish the proof, it suffices to show that $\lambda_2^+>0$ and $\lambda_2^->0$, which is equivalent to showing that $\ker L_+$ is empty and that $\ker L_-$ is spanned by $Q$. These two assertions about the kernel were proved in \cite[Lemma 2.1(ii) and Proposition 5]{dBF05}.\end{proof}

We can also prove the following:
\begin{proposition} Let $\sigma(\mathcal{L})$ be the spectrum of the operator $\mathcal{\mathcal{L}}$, defined in $L^2(\Real^3)\times L^2(\Real^3)$ with domain $H^2(\Real^3)\times H^2(\Real^3)$, and let $\sigma_{ess}(\mathcal{L})$ be its essential spectrum. Then
\begin{equation}
\sigma_{ess}(\mathcal{L}) = \{i y:\,\, y \in \Real, |y| \geq 1\},\quad \sigma \cap \Real = \{-e_0,0,e_0\} \quad \text{with } e_0 >0. 
\end{equation}
Moreover, $e_0$ and $-e_0$ are simple eigenvalues  of $\mathcal{L}$ with eigenfunctions $\mathcal{Y}_+$ and $\mathcal{Y}_- = \overline{\mathcal{Y}}_+$, respectively. The kernel of $\mathcal{L}$ is spanned by $iQ$.
\end{proposition}

\begin{proof}  From Proposition~\ref{P:coercive}, we see that $L_-$ is non-negative. Since it is also self-adjoint, it has a unique square root $L_-^{\frac{1}{2}}$ with domain $H^1$. We show that the self-adjoint operator $P:= L_-^{\frac{1}{2}}L_+L_-^{\frac{1}{2}}$ on $L^2$ with domain $H^4$ has a unique negative eigenvalue. Indeed, consider
\begin{equation}
Z := \tfrac{3-b}{2} Q+x\cdot \nabla Q.
\end{equation}
Then we have that $Z \in H^2$, $Z \in \{Q\}^\perp$, and
\begin{equation}\label{Z_positive_dir} 
(L_+Z, Z) = -\tfrac{2-(1-b)^2}{1-b}\int Q^{2}<0.
\end{equation}
Defining $ h:= L_-^{-\frac{1}{2}}Z \in \{Q\}^{\perp}$, one also has
\begin{equation}
   (Ph,h) = (L_+Z,Z) < 0.
\end{equation}

Hence, by the mini-max principle and an approximation argument, $P$ has a negative eigenvalue $-e_0^2$ and an associated eigenfunction $g$. Defining $\mathcal{Y}_1 := L_-^{\frac{1}{2}}g$, $\mathcal{Y}_2 := \frac{1}{e_0} L_+ \mathcal{Y}_1$, and $\mathcal{Y}_{\pm}:= \mathcal{Y}_{1} \pm i \mathcal{Y}_{2}$, we have $\mathcal{L}\mathcal{Y}_{\pm}=\pm e_0 \mathcal{Y}_{\pm}$.
Uniqueness of the negative direction of $P$ follows from the non-negativity of $L_+$ on $\{\Delta Q\}^\perp$.  The assertions about the kernel of $\mathcal{L}$ follow from the coercivity given by Proposition~\ref{P:coercive}.\end{proof}

{It is also convenient to define a linear form associated to $\mathcal{L}$, namely
\begin{equation}\label{BFG}
B(f,g):=\tfrac{1}{2}(L_+f_1,g_1)+\tfrac{1}{2}(L_-f_2,g_2) = \tfrac{1}{2} \Im(\mathcal{L}f,g),
\end{equation}
as well as the corresponding bilinear form
\begin{equation}
    \Phi(h) := B(h,h).
\end{equation}

We now employ a co-dimensional counting argument to prove coercivity in a slightly different subspace, which will be needed in Section~\ref{Sec:bootstrap}.

\begin{corollary}\label{C:coerc_G_perp} Let $\tilde{G}^\perp$ be the subspace of all $v \in H^1$ (seen as a real vector space) such that
\begin{equation}
    (v_2,Q) = B(v,\mathcal{Y}_+)=B(v,\mathcal{Y}_-)=0.
\end{equation}
Then, for all $v \in \tilde G^\perp$, we have $\Phi(v)\gtrsim\|v\|_{H^1}^2$.
\end{corollary}
\begin{proof}
By the characterization of the spectrum of $\mathcal{L}$ given above, it is enough to guarantee strict positivity of of the quadratic form associated to $B$. Suppose that there exists $h\in \tilde G^{\perp}\backslash\{0\}$ such that $\Phi(h)\leq 0$. We claim that this implies that the set $\{iQ,\mathcal{Y}_+,h\}$ is linearly independent, Indeed, let $a,b,c\in \mathbb R$ be such that
\begin{equation}
    aiQ + b\mathcal{Y}_++ch =0.
\end{equation}
We note that $B(\mathcal{Y}_+,\mathcal{Y}_-) = -e_0(L_-\mathcal{Y}_2,\mathcal{Y}_2) \neq 0$ and that $B(iQ,\cdot)\equiv 0$. Hence, by the definition of $\tilde{G}^\perp$, we have that $bB(\mathcal{Y}_+,\mathcal{Y}_-)=0$. Now, since $iQ$ and $h$ are orthogonal in (the real space) $L^2$, we also have $a = c = 0$. 

We now see that $\Phi$ is non-positive on a subspace of dimension $3$, which contradicts Proposition~\ref{P:coercive}. \end{proof}}

Finally, we establish some decay and regularity properties of functions related to the spectrum of $\mathcal{L}$.  

\begin{lemma}[Spectral decay and regularity]  \label{spectral_decay} 

If $\mathcal{Y}_{\pm} \in H^2(\mathbb R^3)$ are the eigenfunctions to $\mathcal{L}$ corresponding to the real eigenvalues $\pm e_0$ and $\phi \in C^\infty_c(\mathbb{R}^3 \backslash \{0\})$, then
\begin{equation}\label{spectral_control_eigen}
\|\phi(\tfrac{x}{R}) \mathcal{Y}_{\pm}\|_{H^k} \lesssim_{\phi,k,l}  \tfrac{1}{R^l}.
\end{equation}
Moreover, $\mathcal{Y}_{\pm} \in {W}^{3,\frac{6}{5}}(\mathbb{R}^3)$.

If $\lambda \in \mathbb{R}\backslash \sigma(\mathcal{L})$ and $F \in L^2(\mathbb{R}^3)$ is such that 
\begin{equation}
\|\psi(\tfrac{x}{R}) F\|_{H^k} \lesssim_{\psi,k,l}  \tfrac{1}{R^l},
\end{equation}
for any $\psi \in C^\infty_c(\mathbb{R}^3 \backslash \{0\})$, then the solution $f \in H^2(\mathbb{R}^3)$ to 
\begin{equation}
\mathcal{L}f - \lambda f = F
\end{equation}
also satisfies
\begin{equation}\label{spectral_control_f}
\|\phi(\tfrac{x}{R}) f\|_{H^k} \lesssim_{\phi,k,l}  \tfrac{1}{R^l}.
\end{equation}
for any $\phi \in C^\infty_c(\mathbb{R}^3 \backslash \{0\})$.
Moreover, if $F \in W^{1,\frac{6}{5}}(\mathbb{R}^3)$, then $f \in {W}^{3,\frac{6}{5}}(\mathbb{R}^3)$.
\end{lemma}

\begin{remark}
By Sobolev embedding, we see that the functions $\mathcal{Y}_\pm$ and $f$ obtained above are bounded, have bounded first derivatives and decay fast at infinity.
\end{remark}
\begin{remark}
The proof below shows that the functions belong to $W^{2,p}$ for any $1 \leq p < 3/b$ and to $W^{3,q}$ for any $1 \leq  q < 3/(b+1)$, if $F$ is smooth enough.  However, in our applications below we will only need the ${W}^{3,\frac{6}{5}}$ estimates.
\end{remark}
\begin{proof}[Proof of Lemma \ref{spectral_decay}]
We only prove the second item, as the first one follows similarly. The elliptic equation $ \mathcal{L}f - \lambda f = F$ implies

\begin{equation}\label{spectral_system_f}
\begin{cases}
(1-\Delta)f_1 = 3 |x|^{-b}Q^2 f_1 + \lambda f_2 + F_2\\
(1-\Delta)f_2 =  |x|^{-b}Q^2 f_2 - \lambda f_1 + F_1.
\end{cases}
\end{equation}
That implies, for any $\phi \in C^\infty_c(\mathbb{R}^3 \backslash \{0\})$,
\begin{align*}
(1-\Delta)(\phi(\tfrac{x}{R})f_1) & = \tfrac{1}{R^2} \Delta \phi(\tfrac{x}{R}) f_1 + \tfrac{1}{R}\nabla \phi(\tfrac{x}{R})\cdot\nabla f_1  \\
& \quad +  3 |x|^{-b}Q^2 \phi(\tfrac{x}{R}) f_1 + \lambda \phi(\tfrac{x}{R}) f_2 + \phi F_2.
\end{align*}
After differentiating the above equation, we get
\begin{align}
(1&-\Delta)\partial_i\left[\phi(\tfrac{x}{R})f_1\right] \\
&= \tfrac{1}{R^2} \Delta \phi(\tfrac{x}{R}) \partial_if_1 +\tfrac{1}{R^3} \Delta \partial_i\phi(\tfrac{x}{R}) f_1  + \tfrac{1}{R}\nabla \phi(\tfrac{x}{R})\cdot\nabla \partial_if_1 \\
&\quad+ \tfrac{1}{R^2}\nabla \partial_i\phi(\tfrac{x}{R})\cdot\nabla f_1
 +  3 |x|^{-b}Q^2 \phi(\tfrac{x}{R}) \partial_if_1 \\
 &\quad- (b+1)  |x|^{-b-1}\tfrac{x_i}{|x|}  Q^2 \phi(\tfrac{x}{R}) f_1  +  6 |x|^{-b}Q \partial_i Q \phi(\tfrac{x}{R}) f_1 \\ 
&\quad +  \tfrac{3}{R} |x|^{-b}Q^2\partial_i\phi(\tfrac{x}{R}) f_1+ \lambda \phi(\tfrac{x}{R}) \partial_if_2 + \tfrac{\lambda}{R} \partial_i\phi(\tfrac{x}{R}) f_2 + \partial_i(\phi F_2),
\end{align}
with similar equations for $f_2$. Using $\phi$ and its derivatives are supported away from the origin, as well as the decay properties of $Q$, we obtain the bound 
\begin{align*}
\|\phi(\tfrac{x}{R})f\|_{H^3} & \lesssim_{\|Q\|_{\infty}^{},\phi}^{}  \|f\|_{H^1} + \|\nabla Q\|_{6}\|f\|_{3} + \|\phi(\tfrac{x}{R})F\|_{H^1}\\
&\lesssim_{\phi,Q} \|f\|_{H^1}+\|\phi(\tfrac{x}{R})F\|_{H^1}. 
\end{align*}
Therefore, equation \eqref{spectral_control_f} holds for $(k,l) = (3,0)$. Now, assuming it holds for $(k,l)$, we can deduce from \eqref{spectral_system_f} that
\begin{equation}\label{spectral_full_f}
[(1-\Delta)^2+\lambda^2]f_1 = |x|^{-b}Q^2L_+ f_1 +3(1-\Delta)[|x|^{-b}Q^2f_1] + L_+ F_2 - \lambda F_1.
\end{equation}
Choose $\tilde{\phi} \in C^{\infty}_c(\mathbb{R}^3 \backslash \{ 0 \})$ to equal $1$ at the support of $\phi$, so that $\partial^{\alpha} \phi  = \tilde{\phi}\, \partial^{\alpha} \phi  
$ for all multi-indices $\alpha$. Since the Fourier symbol of $(1-\Delta)^2+\lambda^2$ is $(1+|\xi|^2)^2+\lambda^2 \sim (1+|\xi|^2)^2$, we can write
\begin{align}
\|\phi_R f_1\|_{H^{k+1}} &\sim \|[(1-\Delta)^2+\lambda^2](\phi_R f_1)\|_{H^{k-3}}
\\&\lesssim \| \phi_R ((1-\Delta)^2+\lambda^2)(\tilde{\phi}_R f_1)\|_{H^{k-3}} \\
& \quad + \| [(1-\Delta)^2 +\lambda^2;\phi_R](\tilde{\phi}_R f_1)\|_{H^{k-3}}.
\end{align}
The first term in the right-hand side is controlled by \eqref{spectral_full_f} and by the (more than) polynomial decay of $Q$ and its derivatives, as well as on the hypothesis on $F$, giving the bound 
\begin{align*}
 \| \phi_R ((1-\Delta)^2+\lambda^2)(\tilde{\phi}_R f_1)\|_{H^{k-3}} &\lesssim \tfrac{1}{R} \|\tilde{\phi}_R f\|_{H^{k-1}} + \|\tilde{\phi}_R F\|_{H^{k-1}}\\
 &\lesssim \tfrac{1}{R} \|\tilde{\phi}_R f\|_{H^{k}} + \tfrac{1}{R^{l+1}}.
\end{align*}
The remaining term is controlled by the commutator estimate 
\begin{equation}
\|[(1-\Delta)^2+\lambda^2;\phi_R]\|_{H^{k-3}\to {H^k}} \lesssim \tfrac{1}{R}. 
\end{equation}

We then have
\begin{equation}
\|\phi_R f_1\|_{H^{k+1}} \lesssim \tfrac{1}{R} \|\tilde{\phi}_R f\|_{H^{k}} + \tfrac{1}{R^{l+1}}.
\end{equation}
The last inequality shows that, if \eqref{spectral_control_f} holds for $(k,l)$, then it also holds for $(k+1,l+1)$. As the proof for $f_2$ is completely analogous, we will omit it here.

With the decay of $f$ in hand, and recalling $H^2(\mathbb{R}^3) \hookrightarrow L^{\infty}(\mathbb{R}^3)$, we assume now that $F \in W^{1,\frac{6}{5}}$, so that by \eqref{spectral_system_f}, we have $(1-\Delta)f \in L^{\frac{6}{5}}$. After differentiating \eqref{spectral_system_f}, we obtain $f \in W^{3,\frac{6}{5}}$, as desired.\end{proof}

\section{Modulation analysis}\label{Sec:modulation}

The goal of this section is to analyze solutions of \eqref{NLS} during the times that they are close to the orbit of the ground state solution $e^{it}Q$. Here we are measuring closeness in terms of the functional
\begin{equation}
\delta(u) := \left| \int |\nabla u|^2 - \int |\nabla Q|^2\right|.
\end{equation}
When considering a fixed solution $u(t)$ to \eqref{NLS}, we will often abbreviate $\delta(u(t))$ by $\delta(t)$.  Later, we will see that (under the constraint $E(u_0)=E(Q)$), the functional $\delta(u)$ is proportional to the virial functional evaluated at $u$, which explains in part why this functional is so central to the analysis (cf. \eqref{hello-delta} below).

We fix $\delta_0>0$, which will need to be chosen sufficiently small in what follows.  The main result of this section is the following: 

\begin{proposition}[Modulation]\label{P:modulation} Let $u:I\times\R^3\to\C$ be a solution to \eqref{NLS} satisfying $M(u_0)=M(Q)$ and $E(u_0)=E(Q)$.  Let
\[
I_0=\{t\in I: \delta(u(t))<\delta_0\}.
\]
If $\delta_0$ is sufficiently small, then there exist $\alpha:I_0\to\R$, $\theta:I_0\to\R$, and $h:I_0\to H^1$ such that
\[
 u(t) = e^{i\theta(t)}[(1+\alpha(t))Q+h(t)],
\]
with
\begin{equation}\label{modulation-bds}
|\alpha(t)| \sim \|h(t)\|_{H^1_x} \sim \delta(t)\qtq{and}     |\alpha'(t)| +|\theta'(t)-1| \lesssim \delta(t).
\end{equation}
\end{proposition}

We begin with the following non-quantitative result.

\begin{lemma}\label{L:modulation} For any $\eps>0$, there exists $\delta_0>0$ such that for any $v\in H^1$ with $M(v)=M(Q)$ and $E(v)=E(Q)$,  
\[
\delta(v)<\delta_0\implies \inf_{\theta\in\R} \|v-e^{i\theta}Q\|_{H^1}<\eps. 
\]
\end{lemma}

\begin{proof} It suffices to show that for any sequence $v_n\in H^1$ satisfying $M(v_n)\equiv M(Q)$, $E(v_n)\equiv E(Q)$, and $\delta(v_n)\to 0$, there exists $\theta_0\in\R$ such that
\[
\lim_{n\to\infty} \|v_n-e^{i\theta_0}Q\|_{H^1} = 0
\]
along a subsequence.  For this, we use concentration compactness.  In particular, choosing $v_n$ as above and passing to a subsequence, we write $v_n$ in a linear profile decomposition adapted to \eqref{sharpGN}:
\[
v_n = \sum_{j=1}^J \varphi^j(\cdot-x_n^j) + r_n^J
\]
for $1\leq J\leq J^*$, with
\begin{equation}\label{ccpote}
\lim_{J\to J^*}\lim_{n\to\infty} \| |x|^{-b} |r_n^J|^4\|_{L^1} =0.
\end{equation}
By construction, we have decoupling of the $L^2$ and $\dot H^1$ norms, along with the `potential energy' quantity appearing in \eqref{ccpote}. We may also assume that for each $j$, either $x_n^j\equiv 0$ or $|x_n^j|\to\infty$. 

We first observe that we must have $J^*\geq 1$, for otherwise using decoupling, \eqref{ccpote}, and the fact that $E(v_n)\equiv E(Q)$, we would obtain that $\| |x|^{-b} |Q|^4\|_{L^1}=0$. We next claim that $J^*=1$.  To see this, we observe that by the decoupling, nesting of $\ell^p$ spaces, and the sharp Gagliardo--Nirenberg inequality \eqref{sharpGN},
\begin{align}
\||x|^{-b}&  Q^4\|_{L^1} - o_n(1) \nonumber\\
& \leq \sum_{j=1}^{J^*}\||x|^{-b}\varphi(\cdot-x_n^j)\|_{L^1} \nonumber\\
&\leq C_{GN} \sum_{j=1}^{J^*} \|\varphi^j\|_{L^2}^{1-b}\|\nabla\varphi^j\|_{L^2}^{3+b} \leq C_{GN}\|Q\|_{L^2}^{1-b}\biggl[\sum_{j=1}^{J^*}\|\nabla\varphi^j\|_{L^2}^2\biggr]^{\frac{3+b}{2}} \label{super!}\\
& \leq C_{GN}\|Q\|_{L^2}^{1-b}\|\nabla Q\|_{L^2}^{3+b}.\nonumber
\end{align}

Sending $n\to\infty$ and using the fact that $Q$ optimizes \eqref{sharpGN}, we see that each inequality above is actually an equality.  In particular, equality  \eqref{super!} guarantees that $J^*=1$. 

Our decomposition therefore reduces to the form
\[
v_n=\varphi(x-x_n)+r_n.
\]

We can preclude the possibility that $|x_n|\to\infty$ by noting that in this case, we would have
\[
\||x|^{-b}|\varphi(x-x_n)|^4\|_{L^1} \to 0,
\]
which would yield the contradiction $\||x|^{-b}Q^4\|_{L^1}=0$.  Finally, if $r_n$ does not converge to zero strongly in $H^1$, then we can estimate as above to obtain
\[
\| |x|^{-b}Q^4\|_{L^4}\leq (1-\eta)C_{GN}\|Q\|_{L^2}^{1-b}\|\nabla Q\|_{L^2}^{3+b}
\]
for some $\eta>0$, contradicting that $Q$ optimizes \eqref{sharpGN}.  We therefore conclude that $r_n\to 0$ strongly in $H^1$, which then implies that $\varphi$ is an optimizer of \eqref{sharpGN}.  Thus $\varphi=e^{i\theta_0}Q$ for some $\theta_0\in\R$, and $v_n$ converges strongly to $\varphi$ in $H^1$.\end{proof}

We turn to the proof of Proposition~\ref{P:modulation}.  The idea is to use Lemma~\ref{L:modulation} to obtain an initial decomposition of $u(t)$ around $Q$, and then to use the implicit function theorem to obtain a choice of modulation parameters that impose the orthogonality conditions appearing in Proposition~\ref{P:coercive}.  With this choice, we will be able to use the mass and energy constraints to derive the estimates appearing \eqref{modulation-bds}.

\begin{proof}[Proof of Proposition~\ref{P:modulation}] Let $\eps>0$ and choose $\delta_0$ as in Lemma~\ref{L:modulation}.  

Using Lemma~\ref{L:modulation}, we may first define $\theta_0:I_0\to\R$ such that
\[
\|u(t)-e^{i\theta_0(t)}Q\|_{H^1} < \eps \qtq{for all}t\in I_0. 
\]

We will now modify $\theta(t)$ in order to impose an orthogonality condition.  We define
\[
\Phi:H^1\times\R\to\R \qtq{by}\Phi(v,\theta) = \Im(v, e^{i\theta}Q)_{L^2}  
\]
and let
\[
(v_0,\theta_0)=(v_0(t),\theta_0(t))=(e^{i\theta_0(t)}Q,\theta_0(t)). 
\]

Now observe that 
\[
\Phi(v_0,\theta_0)\equiv 0 \qtq{and}\tfrac{\partial \Phi}{\partial\theta}\big|_{(v_0,\theta_0)} \equiv -\|Q\|_{L^2}^2.
\]

We may therefore apply the implicit function theorem to the family of zeros $(v_0(t),\theta_0(t))$: choosing $\eta=\eta(Q)>0$ and $\eps=\eps(\eta)>0$ sufficiently small, we have for each $t\in I_0$ a function 
\[
a_t:B_\eps(e^{i\theta_0(t)}Q)\subset H^1\to B_\eta(\theta_0(t)) \subset\R
\]
such that
\[
\Phi(v,a_t(v)) = 0 \qtq{for all}v\in B_\eps(e^{i\theta_0(t)}Q).
\]

As $u(t)\in B_\eps(e^{i\theta_0(t)}Q)$ for $t\in I_0$, we may therefore define $\theta(t)=a_t(u(t))$ and (using $|\theta(t)-\theta_0(t)|<\eta$) obtain
\[
\Im(u(t),e^{i\theta(t)}Q) = 0 \qtq{and} \|u(t)-e^{i\theta(t)}Q\|_{H^1} < \eta \ll 1. 
\]

Now set 
\[
g(t) = g_1(t)+ig_2(t) = e^{-i\theta(t)}u(t) - Q,\qtq{so that}( g_2(t),Q) \equiv 0. 
\]

We further define $h(t)$ via
\[
g(t) = \alpha(t) Q + h(t),\qtq{where} \alpha=\frac{(g_1,\Delta Q)}{(Q,\Delta Q)} \in \R.
\]

It follows that
\[
u(t) = e^{i\theta(t)}[(1+\alpha(t))Q + h(t)],
\]
with $h$ satisfying the orthogonality conditions appearing in Proposition~\ref{P:coercive}, namely, 
\begin{equation}\label{mod-the-orthog}
(h_1(t),\Delta Q)\equiv 0 \qtq{and} (h_2(t),Q)\equiv 0. 
\end{equation}

To complete the proof of Proposition~\ref{P:modulation}, it therefore remains to prove the bounds appearing in \eqref{modulation-bds}.  We begin by using the fact that $E(u)=E(Q)$ and $M(u)=M(Q)$, gauge invariance, and the identity 
\[
\int g_1 Q + \nabla g_1\cdot \nabla Q - |x|^{-b} g_1 Q^3\,dx = 0
\]
to write
\begin{align*}
0=[E&(u)+M(u)]-[E(Q)+M(Q)] \\
& = \tfrac12(L_+ g_1,g_1)+\tfrac12(L_- g_2,g_2) \\
& \quad -\int |x|^{-b}[\tfrac14|g|^4 + Q g_1|g|^2]\,dx
\end{align*}

We now note that by Proposition~\ref{P:coercive} and the condition $(g_2,Q)\equiv 0$, 
\[
(L_- g_2, g_2) \gtrsim \|g_2\|_{H^1}^2.
\]

On the other hand,
\begin{align*}
(L_+ g_1,g_1) & = (\alpha L_+ Q+L_+ h_1,\alpha Q+h_1) \\
&= \alpha^2(L_+Q,Q) + 2\alpha(L_+Q, h_1) + (L_+ h_1, h_1) 
\end{align*}

We will now combine the last three displays: using Proposition~\ref{P:coercive} and $(h_1,\Delta Q)\equiv 0$, Cauchy--Schwarz, the definition of $g(t)$, and the fact that $\|g\|_{H^1}\ll 1$, we obtain 
\begin{equation}\label{first-bd-on-h}
\begin{aligned}
\|h\|_{H^1}^2 & \lesssim \alpha^2+\|g\|_{H^1}^3+\|g\|_{H^1}^4 + |\alpha(L_+Q, h_1)| \\
& \lesssim \alpha^2 + |\alpha|^3 + \|h\|_{H^1}^3 + |\alpha|\,|(L_+Q,h_1)|. 
\end{aligned}
\end{equation}

To estimate the inner product, we first recall that 
\[
L_+ Q = Q-\Delta Q - 3|x|^{-b}Q^3 = -2|x|^{-b}Q^3 \qtq{and} (h_1,\Delta Q) \equiv 0, 
\]
so that
\begin{align*}
(Q-\Delta Q - 3|x|^{-b}Q^3, h_1) = (-2|x|^{-b}Q^3,h_1)&\implies (|x|^{-b}Q^3,h_1)=(Q,h_1) \\
& \implies (L_+Q,h_1) = -2(Q,h_1).
\end{align*}

To estimate the inner product $(Q,h_1)$, we use the mass constraint:
\begin{equation}\label{use-mass-constraint}
M(Q)=M((1+\alpha)Q+h)\implies 2(Q,h_1)=(2\alpha+\alpha^2)\|Q\|_{L^2}^2 - \|h\|_{L^2}^2.
\end{equation}

Continuing from \eqref{first-bd-on-h}, we find
\[
\|h\|_{H^1}^2 \lesssim \alpha^2+|\alpha|^3 + \|h\|_{H^1}^3. 
\]
Since
\[
|\alpha| \lesssim \|g\|_{H^1}\ll1 \qtq{and} \|h\|_{H^1}\lesssim |\alpha|+\|g\|_{H^1} \ll 1,
\]
this implies
\[
\|h\|_{H^1}\lesssim |\alpha|. 
\]

Returning to \eqref{use-mass-constraint}, we also deduce that $|\alpha|\lesssim \|h\|_{H^1}$, so that we have now obtained
\[
\|h(t)\|_{H^1}\sim |\alpha(t)|
\]
for all $t\in I_0$.  To relate these quantities to $\delta(t)$, we expand
\begin{align*}
\delta(t) &= \biggl| \int |\nabla u|^2 - |\nabla Q|^2 \,dx \biggr| \\
& = \biggl| \int (2\alpha+\alpha^2)|\nabla Q|^2 + |\nabla h|^2 \,dx\biggr| = |2\alpha|\int|\nabla Q|^2\,dx + \mathcal{O}(\alpha^2),
\end{align*}
which (since $|\alpha|\ll1$) implies
\[
 |\alpha(t)| \sim \delta(t). 
\]

To complete the proof of \eqref{modulation-bds}, it remains to estimate $|\alpha'|$ and $|\theta'-1|$. Setting $f(z)=|x|^{-b}|z|^2 z$, we use the identity $g(t)=e^{-i\theta(t)}u(t)-Q$ and the equations for $u$ and $Q$ to obtain
\begin{align*}
i\partial_t g + \Delta g - \dot\theta g \ +(1-\dot\theta)Q + f(e^{-i\theta}u)-f(Q) = 0,
\end{align*}
which we may interpret as an equation in $H^{-1}$. We now multiply this equation by $Q$, integrate, and take the real part.  Recalling the orthogonality conditions \eqref{mod-the-orthog}, we obtain
\begin{align*}
(\dot\theta-1)\|Q\|_{L^2}^2 & = \Re(i\partial_t g,Q) -\dot\theta\Re(g,Q)+ \Re(g,\Delta Q)+\Re(f(e^{-i\theta}u)-f(Q),Q)  \\
& = -\tfrac{d}{dt}\Im(g,Q) -\dot\theta\Re(g,Q)+ \alpha(Q,\Delta Q) + \Re(f(e^{-i\theta}u)-f(Q),Q) \\
& = -\dot\theta\Re(g,Q)+\alpha(Q,\Delta Q) + \Re(f(e^{-i\theta}u)-f(Q),Q).
\end{align*}
Using $\|g\|_{H^1}\ll 1$, this immediately implies that $|\dot\theta|\lesssim 1$, which in turn implies that $\dot\theta\Re(g,Q)=\mathcal{O}(\delta(t))$. As we also have that 
\[
f(e^{-i\theta}u)-f(Q) =|x|^{-b}\cdot \mathcal{O}_{u,Q}(g)\qtq{and} \|g\|_{H^1}\lesssim |\alpha|+\|h\|_{H^1}\lesssim\delta,
\]
we obtain that $|\dot\theta(t)-1|\lesssim \delta(t)$ as desired.  Next, we use the definition of $\alpha$ and the orthogonality condition for $h_1$ in \eqref{mod-the-orthog} to obtain
\begin{align*}
\dot\alpha(Q,\Delta Q) &= \Re(\partial_t g,\Delta Q) \\
& = -\Im(\Delta g,\Delta Q)+\dot\theta\Im(g,\Delta Q) - \Im(f(e^{-i\theta}u)-f(Q),\Delta Q). 
\end{align*}
We now claim that integrating by parts in the first term, estimating as above, and using $|\dot\theta| \lesssim 1$, we can obtain $|\dot\alpha|\lesssim \delta(t),$ as desired.  In fact, the only difficult term is the first one.  Using the equation for $Q$ and integrating by parts, we find that the worst term to estimate will be of the form
\begin{align*}
|\langle \nabla g, |x|^{-b-1} Q^3\rangle| & \lesssim \|\nabla g\|_{L^2} \| |x|^{-(b+1)} Q^3\|_{L^{2}},
\end{align*}
where the second term is finite provided $2(b+1)<3$, i.e. $b<\tfrac12$. 
\end{proof} 

To close the section, we record the following corollary of Proposition~\ref{P:modulation}.

\begin{corollary}\label{const_modul} Let $u$ be a forward-global solution to \eqref{NLS} such that $M(u) = M(Q)$ and $E(u) = E(Q)$. If \begin{equation}\label{delta_int_exp}
 \int_0^\infty \delta(s) \, ds < \infty,
\end{equation}
then $\lim_{t\to\infty}\delta(t)=0$ and there exists $\theta_0\in \R$ such that 
\begin{equation}\label{const_mod_param}
\|u(t) - e^{i(t+\theta_0)}Q\|_{H^1} \lesssim  \int_t^\infty \delta(s)\, ds
\end{equation}
for all $t$ sufficiently large.
\end{corollary}

\begin{proof} We first show that $\lim_{t\to\infty}\delta(t)=0$.  To see this, first observe that \eqref{delta_int_exp} implies that there exists an increasing sequence $t_n\to\infty$ such that $\delta(t_n)\to 0$.  If $\delta(t)\not\to 0$, then we may find an $\eps>0$ and a sequence $t_n'\to\infty$ such that (i) $t_n<t_n'$ for each $n$, (ii) $[t_n,t_n']\subset I_0$ for each $n$, and (iii) $\delta(t_n')>\eps$ for each $n$.  We then note that by Proposition~\ref{P:modulation} and \eqref{delta_int_exp} we have
\[
|\alpha(t_n')-\alpha(t_n)|\lesssim \int_{t_n}^{t_n'} \delta(t)\,dt\to 0 \qtq{as}n\to\infty.
\]
We now observe that $|\alpha(t_n)|\lesssim\delta(t_n)\to 0$, so that $\alpha(t_n')\to 0$ as well.  Applying Proposition~\ref{P:modulation} once more, we deduce that $\delta(t_n')\lesssim|\alpha(t_n')|\to 0$, contradicting the uniform lower bound in (iii).  

With the convergence $\delta(t)\to 0$ in place, we can now assert that $\delta(t)<\delta_0$ for all $t$ sufficiently large.  Then, using Proposition~\ref{P:modulation}, the assumption that $\delta(t)\to 0$, and the fundamental theorem of calculus, we have 
\begin{equation}\label{alpha-int-bd}
\lim_{t\to\infty}\alpha(t) = 0 \qtq{and}|\alpha(t)|\lesssim \int_t^\infty \delta(s)\,ds. 
\end{equation}
Similarly, there exists $\theta_0\in\R$ such that
\[
\lim_{t\to\infty}[\theta(t)-t]=\theta_0,\qtq{with}|\theta(t)-t-\theta_0|\lesssim \int_t^\infty\delta(s)\,ds. 
\]
Thus, using Proposition~\ref{P:modulation}, we deduce 
\begin{align*}
\|u(t)-e^{i(t+\theta_0)}Q\|_{H^1} & = \|u(t) - (1+\alpha(t))e^{i\theta(t)}Q\|_{H^1}+\mathcal{O}\biggl(\int_t^\infty \delta(s)\,ds\biggr) \\
& \lesssim \int_t^\infty \delta(s)\,ds.
\end{align*}
\end{proof}

\section{Nonscattering constrained solutions converge to the ground state}\label{Sec:constrained}

In this section, we prove that if $u$ is a threshold solution with constrained gradient that fails to scatter as $t\to\infty$, then $u$ converges exponentially to the ground state solution as $t\to\infty$.  The proof consists of two steps: (i) compactness for nonscattering constrained solutions, and (ii) convergence for compact constrained solutions. 

\subsection{Nonscattering constrained solutions are compact}

In this section, we show that if $u$ is a threshold solution with constrained gradient that fails to scatter as $t\to\infty$, then the orbit of $u$ for $t\in[0,\infty)$ is pre-compact in $H^1$.  By scaling, we may replace the assumptions on the mass-energy by assumptions on the mass and energy separately. 

\begin{proposition}\label{P:compactness} Suppose $M(u_0)=M(Q)$, $E(u_0)=E(Q)$, and $\|\nabla u_0\|_{L^2}<\|\nabla Q\|_{L^2}$.  Let $u:\R\times\R^3\to\C$ denote the corresponding solution to \eqref{NLS}, which is guaranteed to be global and uniformly bounded in $H^1$ by Proposition~\ref{P:GWP}.  If
\begin{equation}\label{SNB}
\|u\|_{S([0,\infty))}=+\infty,
\end{equation}
then
\[
\{u(t):t\in [0,\infty)\} \qtq{is pre-compact in}H^1. 
\]
The analogous claims hold backward in time. 
\end{proposition}

\begin{proof} The argument is fairly standard, so we will keep our presentation brief. 

Let $\{t_n\}$ be an arbitrary sequence in $[0,\infty)$; without loss of generality, we may assume $t_n\to\infty$.  We apply a profile decomposition to the $H^1$-bounded sequence $\{u(t_n)\}$ to obtain (along a subsequence)
\[
u(t_n) = \sum_{j=1}^J e^{i\tau_n^j\Delta}\varphi^j(\cdot-x_n^j) + r_n^J,\quad 0\leq J\leq J^*\in\{1,2,\dots,\infty\},
\]
where the $\varphi^j$ are nonzero profiles in $H^1$, the parameters $(\tau_n^j,x_n^j)$ satisfy asymptotic orthogonality, the mass and energy decouple appropriately, and the remainder vanishes in the sense that
\begin{equation}\label{LPD-vanishing}
\lim_{J\to J^*}\lim_{n\to\infty} \|e^{it\Delta}r_n^J\|_{S([0,\infty))}=0.
\end{equation}
We may also assume that either $x_n^j\equiv 0$ or $|x_n^j|\to\infty$, and similarly $\tau_n^j\equiv 0$ or $|\tau_n^j|\to\infty$. 

There are three possible scenarios: (i) vanshing (i.e. $J^*=0$), (ii) dichotomy (i.e $J^*\geq 2$), or (iii) compactness (i.e. $J^*=1$).

(i) If vanishing occurs, then
\[
\lim_{n\to\infty}\|e^{it\Delta}u(t_n)\|_{S([0,\infty)}\to 0.
\]
By the stability theory, this implies 
\[
\|u(t+t_n)\|_{S([0,\infty))}=\|u\|_{S((t_n,\infty))}\lesssim 1
\]
for all sufficiently large $n$, contradicting \eqref{SNB}. 

(ii) If dichotomy occurs, then we can use the mass-energy decoupling to show that each $\varphi^j$ satisfies the subthreshold assumption.  Then for each $j$, we can construct a scattering solution $v_n^j$ to \eqref{NLS}.  In particular, if $x_n^j\equiv 0$ and $\tau_n^j\equiv 0$, we let $v_n^j=v^j$ be the scattering solution with initial data $\varphi^j$; if $x_n^j\equiv 0$ and $\tau_n^j\to\pm\infty$, we let $v^j$ be the solution that scatters to $\varphi^j$ as $t\to\pm\infty$ and set $v_n^j(t)=v^j(t+\tau_n^j)$.  If $|x_n^j|\to\infty$, then we can construct a scattering solution $v_n^j$ to \eqref{NLS} with $v_n^j(0)=e^{i\tau_n^j\Delta}\varphi^j(x-x_n^j)$ via the argument of \cite[Proposition~3.2]{CFGM}; in particular, this uses approximation by the \emph{linear} Schr\"odinger equation. 

We then define an approximate solution to \eqref{NLS} by setting
\[
u_n^J(t)=\sum_{j=1}^J v_n^j(t) + e^{it\Delta} r_n^J. 
\]
Then, by construction, we have that for each $J$, $u_n^J(0)-u(t_n)\to 0$ in $H^1$ as $n\to\infty$.  Furthermore, we claim:
\begin{align}
&\limsup_{J\to J^*}\limsup_{n\to\infty}\bigl\{ \|u_n^J(0)\|_{H^1}+\|u_n^J\|_{S([0,\infty))}\bigr\} \lesssim 1, \label{unJbd} \\
&\limsup_{J\to J^*}\limsup_{n\to\infty}\||\nabla|^{s_c}[(i\partial_t+\Delta)u_n^J + |x|^{-b}|u_n^J|^2u_n^J]\|_{L_t^2 L_x^{6/5}([0,\infty)\times\R^3)} = 0. \label{unJsoln}
\end{align}
The estimates \eqref{unJbd} and \eqref{unJsoln} together with the stability result imply that $\|u\|_{S([0,\infty)}<\infty$, contradicting \eqref{SNB} and ruling out the possibility of dichotomy.  Thus it remains to establish \eqref{unJbd} and \eqref{unJsoln}.

The key ingredients for \eqref{unJbd} and \eqref{unJsoln} are the following (a) asymptotic orthogonality of the profiles, (b) the fact that each $v_n^j$ is a scattering solution to \eqref{NLS}, and (c) the vanishing of the remainder \eqref{LPD-vanishing}. For example, using the space-time bounds for $v_n^j$ and approximation by $C_c^\infty$ functions, the orthogonality of parameters implies
\begin{equation}\label{orthogonalitybd}
\lim_{n\to\infty} \{ \|v_n^j v_n^k\|_{L_t^2 L_x^{\frac{3}{1-b}}} + \|v_n^j\,|\nabla|^{s_c} v_n^k\|_{L_t^2 L_x^{\frac{6}{1-b}}} \}= 0 \qtq{for}j\neq k.
\end{equation}
Using this together with the $H^1$ decoupling and Strichartz, we can transfer the estimates for the $v_n^j$ to the entire approximate solution $u_n^J$, yielding \eqref{unJbd}.  For \eqref{unJsoln}, we denote $|x|^{-b}|z|^2 z$ by $F(z)$ and observe that
\begin{align}
(i\partial_t+\Delta)u_n^J + F(u_n^J)&= -\biggl[\sum_{j=1}^J F(v_n^j)-F(\sum_{j=1}^J v_n^j)\biggr]\label{unJerror1} \\
& \quad + F(u_n^J-e^{it\Delta}r_n^J)-F(u_n^J).\label{unJerror2}
\end{align}
We then note that (up to complex conjugates) \eqref{unJerror1} can be written as a finite linear combination of terms of the form $v_n^j v_n^k v_n^\ell$, where not all of $j,k,\ell$ are equal; the total number of terms depends on $J$, but this does not matter once one proves
\[
\lim_{n\to\infty} \| |\nabla|^{s_c}[v_n^j v_n^k v_n^\ell]\|_{L_t^2 L_x^{6/5}}=0
\]
for such triples $j,k,\ell$.To prove this, one relies on the orthogonality in the form \eqref{orthogonalitybd}, using the fractional product rule and a paraproduct estimate as in \cite{KenigMerleH12} to deal with the presence of the nonlocal operator $|\nabla|^{s_c}$.  To deal with \eqref{unJerror2}, one observes that the factor $e^{it\Delta}r_n^J$ is present and uses the space-time bounds for $u_n^J$ together with the vanishing \eqref{LPD-vanishing}. 

(iii) We have now shown that vanishing and dichotomy cannot occur, so that compactness ($J^*=1$) is the only remaining option. In particular, we have
\[
u(t_n)=e^{i\tau_n\Delta}[\phi(\cdot-x_n)] + r_n.
\]
The mass and energy decoupling guarantee that $r_n\to 0$ strongly in $H^1$.  Indeed, we have weak convergence by assumption, so that if strong convergence fails, the profile $\phi$ would obey the subthreshold hypothesis.  Then, arguing as we did to prevent dichotomy, we could deduce scattering for $u$, contradicting \eqref{SNB}.  It therefore remains to see that we must have $\tau_n\equiv 0$ and $x_n\equiv 0$.  Indeed, if $|x_n|\to\infty$ then we can argue again as in \cite[Proposition~3.2]{CFGM} to see that the solution $u$ must scatter; similarly, if $\tau_n\to\pm\infty$, then we can use stability theory (comparing $u$ with the linear solution $e^{i(t+\tau_n)\Delta}\phi$) to deduce scattering for $u$.  In particular, both of these possibilities would contradict \eqref{SNB}.  We conclude that $u(t_n)=\phi+o_n(1)$ in $H^1$, yielding compactness as desired. \end{proof}

\subsection{Convergence for compact constrained solutions}  The main goal of this section is to establish the following:

\begin{proposition}\label{P:constrained-converge} Suppose $u:{\mathbb R}\times\R^3\to\C$ is a global solution to \eqref{NLS} satisfying
\[
M(u)=M(Q),\quad E(u)=E(Q),\qtq{and} \|\nabla u_0\|_{L^2}<\|\nabla Q\|_{L^2}.
\]
Suppose further that
\[
\{u(t):t\in[0,\infty)\} \qtq{is pre-compact in} H^1.
\]
Then there exists $C,c>0$ and $\theta_0\in\R$ such that
\[
\|u(t)-e^{i(t+\theta_0)}Q\|_{H^1} \leq Ce^{-ct}\qtq{for all}t\geq 0. 
\]

\end{proposition}

The key to the proof of Proposition~\ref{P:constrained-converge} will be a localized virial estimate that takes the modulation analysis of Section~\ref{Sec:modulation} into account.  This estimate will allow us to control the functional $\delta(t):=\delta(u(t))$.  To obtain the desired convergence to the ground state, we will ultimately rely on Corollary~\ref{const_modul}.

Throughout this section, we assume that $u$ is a solution as in the statement of Proposition~\ref{P:constrained-converge}. 

\begin{lemma}[Virial Estimate]\label{L:virial} There exists $C>0$ such that for any $[t_1,t_2]\in[0,\infty)$, 
\[
\int_{t_1}^{t_2}\delta(t)\,dt\leq C\{\delta(t_1)+\delta(t_2)\}. 
\]
\end{lemma}

\begin{proof} Let $\phi$ be a real-valued, radial function such that
\[
\phi(x)=\begin{cases} |x|^2 & |x|\leq 1, \\ \text{const} & |x|>3,\end{cases}\qtq{and} |\partial^\alpha \phi(x)|\lesssim_\alpha |x|^{2-|\alpha|}.
\]
We also impose that $\partial_r\phi\geq 0$, where $\partial_r=\nabla\cdot\tfrac{x}{|x|}$ is the radial derivative. 

Given $R\geq 1$, let
\[
w_R(x)=R^2\phi(\tfrac{x}{R}),\quad w_\infty(x)=|x|^2, 
\]
and define for $R\in[1,\infty]$ the functional
\[
P_R[u] = 2\Im\int_{\R^3}\bar u \nabla u \cdot\nabla w_R\,dx. 
\]
Then a direct computation using \eqref{NLS} and integration by parts yields
\begin{equation}\label{virial-identity}
\tfrac{d}{dt}P_R[u(t)] = F_R[u(t)],
\end{equation}
where
\begin{align*}
F_R[u] &= \int (-\Delta\Delta w_R)|u|^2 + 4\Re \bar u_j u_k\partial_{jk}[w_R] \\
&\quad\quad -|x|^{-b}|u|^4\Delta w_R - b|x|^{-b-2}|u|^4 x\cdot\nabla w_R\,dx. 
\end{align*}
In the case $R=\infty$, we use the fact that $E(u)=E(Q)$ and the identity \eqref{energy-of-Q} for $E(Q)$ to write
\begin{equation}\label{hello-delta}
\begin{aligned}
F_\infty[u] & = \int 8 |\nabla u|^2 - (6+2b)|x|^{-b}|u|^4\,dx \\
& = 8(3+b)E(u) - 4(1+b)\int |\nabla u|^2\,dx  \\
& =  4(1+b)\delta(t). 
\end{aligned}
\end{equation} 

Next, we claim that
\begin{equation}\label{virial-on-Q}
F_R[e^{i\theta}Q] = 0 \qtq{for all}R\in[1,\infty]\qtq{and}\theta\in\R.
\end{equation}
Indeed, since $Q$ is real-valued, we have
\begin{equation}\label{P_R is zero}
P_R[e^{i(t+\theta)}Q] = 0 \qtq{for all}t\in\R.
\end{equation}
As $e^{i(t+\theta)}Q$ solves \eqref{NLS}, the identity \eqref{virial-identity} implies 
\[
F_R[e^{i(t+\theta)}Q] = 0\qtq{for all}t\in\R,
\]
which (evaluating at $t=0$) implies \eqref{virial-on-Q}. 

Now we fix $R\geq 1$, which will be specified below, and we choose $\delta_1\in(0,\delta_0)$.  We then define
\[
\chi=\chi_{J},\quad J=\{t\in[t_1,t_2]:\delta(t)<\delta_1\},
\]
and denote $\chi^c=1-\chi$. Then \eqref{virial-identity} and \eqref{virial-on-Q} yield
\begin{equation}\label{modulated-virial-identity}
\begin{aligned}
\tfrac{d}{dt}P_R[u] - F_\infty[u] & = \chi^c(t)\bigl\{F_R[u]-F_\infty[u]\bigr\} \\
&\quad + \chi(t)\bigl\{F_R[u]- F_\infty[u] - (F_R[e^{i\theta(t)}Q] -F_\infty[e^{i\theta(t)}Q])\bigr\},
\end{aligned}
\end{equation}
where $\theta(t)$ is as in Proposition~\ref{P:modulation} (and, in particular, is defined on the support of $\chi(t)$).  Our task is now to bound $P_R[u(t_j)]$, as well as the terms on the right-hand-side of \eqref{modulated-virial-identity}.

Fix $j\in\{1,2\}$. If $\delta(t_j)\geq \delta_1$, then we use Cauchy--Schwarz to estimate
\begin{equation}\label{PRtj}
|P_R[u(t_j)]| \lesssim R\|u\|_{L_t^\infty H_x^1}^2 \lesssim \tfrac{R}{\delta_1} \delta(t_j). 
\end{equation}
If instead $\delta(t_j)<\delta_1$, then we use \eqref{P_R is zero} and Proposition~\ref{P:modulation} to estimate
\begin{equation}\label{PRtj2}
\begin{aligned}
|P_R[u(t_j)]| & = \biggl| 2\Im \int\bigl( \bar u \nabla u - e^{-i\theta(t_j)}Q\nabla[e^{i\theta(t_j)}Q]\bigr)\cdot\nabla w_R \,dx\biggr| \\
& \lesssim R\{\|u\|_{L_t^\infty H_x^1}+\|Q\|_{H^1}\}\|u(t_j)-e^{i\theta(t_j)}Q\|_{H^1} \\
& \lesssim R\delta(t_j). 
\end{aligned}
\end{equation}

We turn to the terms on the right-hand side of \eqref{modulated-virial-identity}:

For the $\chi^c$ term, we have $\delta(t)\geq \delta_1$, and we can estimate using $H^1$ pre-compactness.  In particular, we let $\eps>0$ and choose $R$ sufficiently large that
\[
\sup_{t\in[0,\infty)} \int_{|x|>R} |\nabla u|^2 + |x|^{-2}|u|^2 + |x|^{-b} |u|^4\,dx \ll\eps^2. 
\]
We then write
\begin{align*}
F_R&[u]-F_\infty[u] \\
& = - \int_{|x|>R} 8|\nabla u|^2 - (6+2b)|x|^{-b}|u|^4 + 4\Re \bar u_j u_k\partial_{jk}[w_R]\,dx \\
& \quad + \int_{|x|>R} (-\Delta\Delta w_R)|u|^2 - |x|^{-b} |u|^4 \Delta w_R\ - b|x|^{-b-2}|u|^4 x\cdot\nabla w_R\,dx, 
\end{align*}
from which we may deduce that 
\begin{equation}\label{chic-error}
|F_R[u(t)]-F_\infty[u(t)]| < \tfrac{\eps}{\delta_1}\delta(t) \qtq{uniformly for}t\in[t_1,t_2]\backslash J.
\end{equation}

For the $\chi$ term on the right-hand side of \eqref{modulated-virial-identity}, we set $Q(t)=e^{i\theta(t)}Q$ and expand the error term as 
\begin{equation}\label{mod-errors}
\begin{aligned}
& -\int_{|x|>R} 8[|\nabla u|^2-|\nabla Q(t)|^2]-(6+2b)|x|^{-b}[|u|^4-|Q(t)|^4]\,dx \\
& \quad + 4\int_{|x|>R} \Re[\bar u_j u_k - \bar Q_j(t) Q_k(t)]\partial_{jk}w_R -[|u|^2-|Q(t)|^2]\Delta\Delta w_R\,dx  \\
& \quad - \int_{|x|>R} |x|^{-b}[|u|^4-|Q(t)|^4]\Delta w_R  + b|x|^{-b-2}[|u|^4-|Q(t)|^4]\,x\cdot\nabla w_R \,dx.
\end{aligned}
\end{equation}
The key to estimating the terms in \eqref{mod-errors} is to observe that (i) in each term we can exhibit the difference $u(t)-Q(t)$ measured in $H^1$ and (ii) the remaining terms will contain either $u$ or $Q$ at radii $|x|>R$, so that (choosing $R$ possibly even larger depending on $Q$) they are $\mathcal{O}(\eps)$.  For example, using compactness and Proposition~\ref{P:modulation}, we have
\begin{align*}
\| \Re[\bar u_j& u_k - \bar Q_j Q_k]\partial_{jk}w_R\|_{L^1(|x|>R)} \\
& \lesssim \| \nabla[u-Q(t)]\|_{L^2(|x|>R)}^2 \\
& \lesssim \{\|\nabla u\|_{L^2(|x|>R)} + \|\nabla Q\|_{L^2(|x|>R)}\}\|u-Q(t)\|_{H^1} \lesssim \eps\delta(t). 
\end{align*}
The terms containing negative powers of $|x|$ are simpler, in the sense that we can exhibit negative powers of $R$, which can be made to be $\mathcal{O}(\eps)$ directly by choosing $R$ large enough. In particular, we obtain
\begin{equation}\label{chi-error}
|F_R[u]- F_\infty[u] - (F_R[e^{i\theta(t)}Q] -F_\infty[e^{i\theta(t)}Q])|<\eps\delta(t)\qtq{uniformly for}t\in J. 
\end{equation}

We now continue from \eqref{modulated-virial-identity}, integrating over $[t_1,t_2]$ and inserting \eqref{PRtj}, \eqref{PRtj2}, \eqref{hello-delta}, \eqref{chic-error}, and \eqref{chi-error}.  Recalling $\delta_1\ll 1$, we derive that
\[
\int_{t_1}^{t_2} \delta(t) \,dt \leq \tfrac{CR}{\delta_1}[\delta(t_1)+\delta(t_2)] + \tfrac{\eps}{\delta_1} \int_{t_1}^{t_2} \delta(t)\,dt.
\]
Choosing $\eps=\eps(\delta_1)$ sufficiently small, we complete the proof. \end{proof}

Applying the preceding lemma on intervals of the form $[0,T]$ and using boundedness of $\delta(\cdot)$, we immediately obtain the following:

\begin{corollary}\label{C:mod-virial} We have
\[
\int_0^\infty \delta(t)\,dt<\infty.
\]
\end{corollary}

We are now in a position to complete the proof of Proposition~\ref{P:constrained-converge}. 

\begin{proof}[Proof of Proposition~\ref{P:constrained-converge}] We begin by upgrading Corollary~\ref{C:mod-virial} to an exponential estimate.  Using Lemma~\ref{L:virial}, Corollary~\ref{C:mod-virial}, and Corollary~\ref{const_modul} (which yields $\delta(t)\to0$), we obtain
\[
\int_{t}^\infty \delta(s)\,ds\leq C\delta(t) \qtq{for all}t>0. 
\]
By Gronwall's inequality, this implies that
\[
\int_t^\infty \delta(s)\,ds \lesssim e^{-ct} \qtq{for some}c>0\qtq{and all}t>0. 
\]
Finally, by Corollary~\ref{C:mod-virial} and Corollary~\ref{const_modul}, we deduce that there exists $\theta_0\in\R$ such that
\[
\|u(t)-e^{i(t+\theta_0)}Q\|_{H^1} \lesssim e^{-ct} \qtq{for all}t>0. 
\]

\end{proof}
\section{Global unconstrained solutions converge to the ground state}\label{Sec:unconstrained}

In this section, we prove that if $u$ is a threshold solution with unconstrained gradient that exists globally forward in time, then $u$ converges exponentially to the ground state solution as $t\to\infty$.  Our proofs require that we impose some additional localization assumption on $u$ (namely, either $u$ is radial or $xu\in L^2$). 

\begin{proposition}\label{P:global-unconstrained} Suppose $u:[0,\infty)\times\R^3\to\C$ is a radial, forward-global solution to \eqref{NLS} satisfying
\[
M(u)=M(Q),\quad E(u)=E(Q),\qtq{and}\|\nabla u_0\|_{L^2}>\|\nabla Q\|_{L^2}. 
\]
Then there exists $c>0$ and $\theta_0\in\R$ such that
\[
\|u(t)-e^{i(t+\theta_0)}Q\|_{H^1} \lesssim e^{-ct} \qtq{for all}t>0. 
\]
\end{proposition}

\begin{proof}  We break the proof into several steps:

\underline{\emph{Step 1. Modulated virial estimate.}}

We utilize the localized virial identity as in the proof of Lemma~\ref{L:virial} above.  In particular, we recall the notation
\[
P_R[u]= 2\Im\int \bar u\nabla u\cdot\nabla w_R\,dx. 
\]
Fix a time interval $[t_1,t_2]$ and $\delta_1\in(0,\delta_0)$, and define
\[
\chi=\chi_J,\quad J=\{t\in[t_1,t_2]:\delta(t)<\delta_1\},\qtq{and}\chi^c = 1-\chi. 
\]
Adopting the notation from the proof of Lemma~\ref{L:virial}, we fix $R\geq 1$ and obtain the modulated virial identity 
\begin{equation}\label{mod-virial-again}
\begin{aligned}
\tfrac{d}{dt}P_R[u] - F_\infty[u] & = \chi^c(t)\bigl\{F_R[u]-F_\infty[u]\bigr\} \\
&\quad + \chi(t)\bigl\{F_R[u]- F_\infty[u] - (F_R[e^{i\theta(t)}Q] -F_\infty[e^{i\theta(t)}Q])\bigr\},
\end{aligned}
\end{equation}
where $\theta(t)$ is as in Proposition~\ref{P:modulation}.  In the present setting, we also impose that the weight $\phi$ satisfy
\[
|\nabla\phi(x)|\leq 2|x| \qtq{and}|\partial_{jk}\phi|\leq 2\qtq{for all}x.
\]

By the computation in \eqref{hello-delta}, we may write
\[
F_\infty[u]=-4(1+b)\delta(t),
\]
where here we use the condition $\|\nabla u(t)\|_{L^2}>\|\nabla Q\|_{L^2}$.  Our task is then to control the terms on the right-hand side of \eqref{mod-virial-again} by $\delta(t)$.

We first consider the $\chi^c$ term.  We begin by writing 
\begin{align}
F_R[u] - F_\infty[u] & = \int_{|x|>R} [4\Re \bar u_j u_k \partial_{jk}w_R - 8|\nabla u|^2 ]\,dx \label{it's negative} \\
& \quad + \int_{|x|>R} \mathcal{O}(R^{-2} |u|^2 + R^{-b} |u|^4)\,dx \label{LR-error}
\end{align}
By construction, we have that \eqref{it's negative}$\leq 0$. On the support of $\chi^c(t)$, we have $\delta(t)\geq \delta_1$, and hence we have the trivial estimate
\[
R^{-2}\|u\|_{L^2}^2 \lesssim R^{-2}\tfrac{1}{\delta_1}\delta(t). 
\]
For the remaining term in \eqref{LR-error}, we use the radial Sobolev embedding estimate, Young's inequality, the bound $\|\nabla u\|_{L^2}^2\lesssim 1+\delta(t)$, and the fact that $\delta_1<\delta(t)$ to obtain
\begin{align*}
R^{-b}\| u\|_{L^4}^4 & \lesssim R^{-2-b}\|xu\|_{L^\infty}^2 \|u\|_{L^2}^2 \\
& \lesssim R^{-2-b}\|u\|_{L^2}^3 \|\nabla u\|_{L^2} \\
& \lesssim R^{-2-b}[\|u\|_{L^2}^4 +\|\nabla u\|_{L^2}^2] \\
& \lesssim R^{-2-b}[1+\delta(t)] \lesssim R^{-2-b}[\tfrac{1}{\delta_1}+1]\delta(t). 
\end{align*}

We turn to the $\chi$ term in \eqref{mod-virial-again} and write
\begin{align*}
F_R[u]&-F_\infty[u] - (F_R[Q(t)]-F_\infty[Q(t)]) \\
&= \int_{|x|>R} 4\Re\partial_{jk}w_R[\bar u_j u_k -\bar Q_j(t)Q_k(t)] -8 [|\nabla u|^2 - |\nabla Q(t)|^2]\,dx \\
& \quad + \mathcal{O}\biggl[\int_{|x|>R} R^{-2}[|u|^2-|Q(t)|^2] + R^{-b}[|u|^4-|Q(t)|^4]\,dx\biggr]\,dx,
\end{align*}
where we denote $Q(t)=e^{i\theta(t)}Q$. Using Proposition~\ref{P:modulation} and the decay of $Q$, the first line on the right-hand side above may be estimated by
\begin{align*}
\bigl\{\|u(t)-Q(t)\|_{H^1}&+\|Q\|_{H^1(|x|>R)}\bigr\} \|u(t)-Q(t)\|_{H^1} \\
&\lesssim \{\delta(t)+\eta(R)\}\delta(t) \lesssim \{\delta_0+\eta(R)\}\delta(t),
\end{align*}
where here and below we denote
\[
\eta(R):=\|Q\|_{H^1(|x|>R)}. 
\]
For the second line, we use the the fact $u$ is $H^1$-bounded on the support of $\chi$ and obtain an estimate of the form
\[
[R^{-2}+R^{-b}]\delta(t).
\]
Continuing from \eqref{mod-virial-again} and collecting our estimates, we deduce that
\begin{equation}\label{mod-virial-again-conclusion}
\begin{aligned}
\tfrac{d}{dt}P_R[u] & \leq -2c\delta(t) + \mathcal{O}\{R^{-b}+R^{-2}\tfrac{1}{\delta_1}+\delta_1+\eta(R)\}\delta(t)\\
& \leq -c\delta(t)
\end{aligned}
\end{equation}
for some $c>0$, provided we choose $\delta_1$ sufficiently small and $R=R(\delta_1,Q)$ sufficiently large. Integrating over $[t_1,t_2]$, we obtain the virial estimate
\begin{equation}\label{virial-estimate-again}
\int_{t_1}^{t_2}c\delta(t)\,dt\leq P_R[u(t_1)]-P_R[u(t_2)]. 
\end{equation}

\underline{\emph{Step 2. Positivity and upper bounds for $P_R[u]$.}}

A direct computation again using \eqref{NLS} and integration by parts shows that
\[
P_R[u] = \tfrac{d}{dt} \int |u|^2 w_R\,dx. 
\]
Now suppose that $P_R[u(t_0)]\leq 0$ for some $t_0$.  It then follows from \eqref{mod-virial-again-conclusion} that $P_R[u(t)]<-\eta<0$ for all $t>t_0$.  Thus
\[
\int |u(t)|^2 w_R\,dx \leq \int |u(t_0)|^2 w_R\,dx - \eta(t-t_0),
\]
yielding a contradiction for $t$ sufficiently large (as $w_R$ is positive). Thus $P_R[u(t)]$ is positive for all $t\geq 0$. 

Next, we claim that there exists $C>0$ such that
\begin{equation}\label{bound-for-PR}
|P_R[u(t)]| \leq CR \delta(t) \qtq{for all}t\geq 0.
\end{equation}
If $\delta(t)>\delta_1$, this follows from the trivial estimate
\[
|P_R[u(t)]| \lesssim R\|u\|_{H^1}^2\lesssim R[\delta(t)+\|Q\|_{H^1}^2] \lesssim R[1+\delta_1^{-1}\|Q\|_{H^1}^2] \delta(t). 
\]
If instead $\delta(t)\leq \delta_1$, then we use the fact that $P_R[Q(t)]\equiv 0$ (where $Q(t)=e^{i\theta(t)}Q$ as above) to obtain
\begin{align*}
|P_R[u(t)]| & = |P_R[u(t)]-P_R[Q(t)]| \\
& \lesssim \{\|u(t)-Q(t)\|_{H^1}+\|Q\|_{H^1}\}\|u(t)-Q(t)\|_{H^1} \\
& \lesssim \{\delta(t)+1\}\delta(t)\lesssim \delta(t).
\end{align*}

\underline{\emph{Step 3. Limit for $P_R[u(t)]$.}}

In Steps 1 and 2, we have established that $P_R[u(t)]$ is strictly positive and strictly decreasing.  Thus there exists $\ell\geq 0$ such that $\lim_{t\to\infty}P_R[u(t)]\to\ell$.  In particular, by \eqref{virial-estimate-again}, we find that 
\begin{equation}\label{delta-L1}
\int_0^\infty \delta(t)\,dt<\infty.
\end{equation} 
This implies that $\delta(t_n)\to 0$ along some sequence $t_n\to\infty$, which (in light of \eqref{bound-for-PR}) then implies
\begin{equation}\label{PR-to-zero}
\lim_{t\to\infty}P_R[u(t)]=0.
\end{equation}

\underline{\emph{Step 4. Exponential bounds and conclusion of the proof.}}

Using \eqref{PR-to-zero}, \eqref{virial-estimate-again}, and \eqref{bound-for-PR}, we find that
\[
\int_t^\infty \delta(s)\,ds \lesssim \delta(t),
\]
which (by Gronwall's inequality) implies
\[
\int_t^\infty \delta(s)\,ds\lesssim e^{-ct}
\]
for some $c>0$.  Combining this estimate with \eqref{delta-L1} and Corollary~\ref{const_modul}, we conclude that
\[
\|u(t)-e^{i(t+\theta_0)}Q\|_{H^1}\lesssim e^{-ct}
\]
for some $\theta_0$ and all $t$ sufficiently large, which completes the proof of Proposition~\ref{P:global-unconstrained}.\end{proof}

{The proof above actually establishes the following:
\begin{proposition}
\label{P:global-unconstrained_finite_variance}
If $u$ is a solution as in Proposition~\ref{P:global-unconstrained}, then $xu\in L^2$.
\end{proposition}
\begin{proof} We showed above that
\begin{equation}
\tfrac{d}{dt} \int |u|^2 w_R\,dx = P_R[u] \geq 0.
\end{equation}
This implies that $V_R(u(t)):=\int |u(t)|^2 w_R\,dx$ is non-decreasing. As 
\[
\|u(t)-e^{it}Q\|_{H^1}\to 0\qtq{as}t \to \infty,
\]
we have that 
\begin{equation}
V_R(u(0)) \leq V_R(u(t)) \leq V_R(Q) \leq \int |x|^2 Q^2\, dx
\end{equation}
for all $t>0$. In particular, $V_R(u(0))$ is uniformly bounded in $R$. Since it is also non-decreasing, we see that 
\begin{equation}
    \int |x|^2 |u(0,x)|^2\, dx \leq \int |x|^2 Q^2\, dx < \infty.
\end{equation}
\end{proof}

Using similar arguments, we can also establish convergence for data with $xu_0\in L^2$.  As the proof is very similar to that of Proposition~\ref{P:global-unconstrained} (in fact, it is strictly easier, as we do not need to truncate the virial identity), we omit the proof.

\begin{proposition}\label{P:global-unconstrained-2} Suppose $u:[0,\infty)\times\R^3\to\C$ is a solution to \eqref{NLS} such that $xu_0\in L^2$ and
\[
M(u)=M(Q),\quad E(u)=E(Q),\qtq{and}\|\nabla u_0\|_{L^2}>\|\nabla Q\|_{L^2}. 
\]
Then there exists $c>0$ and $\theta_0\in\R$ such that
\[
\|u(t)-e^{i(t+\theta_0)}Q\|_{H^1} \lesssim e^{-ct} \qtq{for all}t>0. 
\]
\end{proposition}

}


\section{Construction of special solutions}\label{Sec:special}

In this section, we show that the mass-energy scattering threshold admits new dynamics compared to the sub-threshold case. In particular, we construct solutions that are different from the ground state solution, but whose distance to $e^{it}Q$ decreases exponentially in time in one direction. To do so, we construct a family of approximate solutions to \eqref{NLS} and then perform a fixed point argument. 
\subsection{Nonlinear estimates} We begin by establishing some nonlinear estimates adapted to the linearized equation.  Recall the notation $K(\cdot)$ and $R(\cdot)$ introduced in \eqref{def:KR} and the function spaces $N(I), Z(I)$ introduced in Section~\ref{S:WPS}.

\begin{lemma}[Preliminary estimates]\label{prel_estimates}
If $I \subset \mathbb{R}$ is such that $|I| \leq 1$, then:
\begin{align}
&\| K(f)\|_{N(I)} \lesssim |I|^{\frac{1}{2}}\| f\|_{Z(I)},\\
&\| R(f)\|_{N(I)} \lesssim \| f\|_{Z(I)}^2+\|f\|_{Z(I)}^3,\\
&\|R(f)-R(g)\|_{N(I)} \lesssim 
\|f-g\|_{Z(I)}\left[\|f\|_{Z(I)} +\|g\|_{Z(I)}+
\|f\|_{Z(I)}^2+\|g\|_{Z(I)}^2\right].\label{sub_P_grad_r_point}
\end{align}
\end{lemma}
\begin{proof}
For the first estimate, we use H\"older's inequality to obtain
\begin{align}\label{first_item_1}
 \| |x|^{-b} Q^2 \langle \nabla \rangle f\|_{L^2_t L^{\frac{6}{5}}_x} &\lesssim \||x|^{-b} Q^2\|_{L^2_t L^3_x}\|\langle \nabla \rangle f\|_{L^\infty_t L^2_x},\\
\label{first_item_2}
    \| |x|^{-b} Q \nabla Q\,f\|_{L^2_t L^{\frac{6}{5}}_x} &\lesssim \||x|^{-b} Q |\nabla Q|\|_{L^2_t L^3_x}\| f\|_{L^\infty_t L^2_x}, \\
\label{first_item_3}
    \| |x|^{-b-1} Q^2 f\|_{L^2_t L^{\frac{6}{5}}_x} &\lesssim \||x|^{-b-1}Q^2\|_{L^2_t L^{\frac{3}{2}}_x} \|f\|_{L^\infty_t L^6_x}.
\end{align}

For the remaining estimates, we use \eqref{stein_lemma} and H\"older to obtain
\begin{align}\label{second_item_1}
    \||x|^{-b}f g h\|_{L^2_t L^{\frac{6}{5}}_x} &\lesssim \| |\nabla|^{\frac{3+2b}{6}}f\|_{L^\infty_t L^2_x} \| |\nabla|^{\frac{3+2b}{6}}g\|_{L^\infty_t L^2_x} \| h\|_{L^2_t L^6_x},\\
    \label{second_item_2}
    \||x|^{-b-1}f g h\|_{L^2_t L^{\frac{6}{5}}_x} &\lesssim \| |\nabla|^{\frac{b+1}{3}}f\|_{L^6_t L^{\frac{18}{7}}_x} \| |\nabla|^{\frac{b+1}{3}}g\|_{L^6_t L^{\frac{18}{7}}_x}\| |\nabla|^{\frac{b+1}{3}}h\|_{L^6_t L^{\frac{18}{7}}_x}
\end{align}
for arbitrary $f,g,h$.  As one has
\begin{equation}\label{diff_R}
    |R(f)-R(g)| \lesssim |x|^{-b}(Q|f| + Q|g| + |f|^2 + |g|^2)|f-g|
\end{equation}
and
\begin{align}
    |\nabla(R(f)-R(g))| &\lesssim |x|^{-b-1}(Q|f| + Q|g| + |f|^2 + |g|^2)|f-g|\\
    \label{grad_diff_R}&\quad + |x|^{-b}(Q|f| + Q|g| + |f|^2 + |g|^2)|\nabla (f-g)|\\
    &\quad + |x|^{-b}[\nabla(Q|f| + Q|g| + |f|^2 + |g|^2)]|f-g|,
\end{align}
the estimates now follow (using the decay properties of $Q$).\end{proof}

\subsection{An approximate family of solutions}
We next construct a family of solutions to equations that successively approximate the linearized equation. In what follows, we utilize the notation for eigenfunctions/eigenvalues of $\mathcal{L}$ from Section~\ref{Sec:spectral}.

\begin{proposition}\label{family_appr}  Let $A \in \mathbb R\backslash\{0\}$. There exists a sequence $\{Z_k^A\}_{k \geq 1}$ of functions in $H^2(\mathbb R^3)\cap {W}^{3,\frac{6}{5}}(\mathbb R^3)$ such that 
\[
Z_1^A = A \mathcal{Y}_+\qtq{and}\mathcal{V}_k^A = \sum_{j=1}^k e^{-je_0 t}Z_j^A\quad(k\geq 1)
\]
satisfy
\begin{equation}\label{eq_sol_v}
\partial_t\mathcal{V}_k^A + \mathcal{LV}_k^A = iR(\mathcal{V}_k^A)+\mathcal{O}\left(e^{-\left(k+1\right)e_0t}\right) \qtq{in}   W^{1,\frac{6}{5}}\qtq{as}t\to\infty.
\end{equation}
\end{proposition}

\begin{proof}
The sequence is constructed inductively. To simplify notation, we omit the superscript $A$ throughout the proof. Define $Z_1 = A\mathcal{Y}_+$ and note that
\begin{equation}
    \partial_t \mathcal{V}_1 + \mathcal{LV}_1 -iR(\mathcal{V}_1) = -iR(\mathcal{V}_1).
\end{equation}
Now note that for any $v$, we have the pointwise bounds
\begin{align*}
|R(v)| &\lesssim |x|^{-b}(Q|v|^2 + |v|^3),\\
|\nabla R(v)| &\lesssim |x|^{-b-1}(Q + |v| )|v|^2 + |x|^{-b}(Q + |v| ) |\nabla v| |v|\\
    &\quad + |x|^{-b}|\nabla(Q|v|+ |v|^2 )|v|.
\end{align*}
Thus we have, by \eqref{second_item_1},
\begin{align*}
\|R(v)\|_{L^{\frac{6}{5}}} &\lesssim (\||\nabla|^{\frac{3+2b}{6}}Q\|_{L^2} +\||\nabla|^{\frac{3+2b}{6}}v\|_{L^2}) \||\nabla|^{\frac{3+2b}{6}}v\|_{L^2} \|v\|_{L^6}  \\
&\lesssim  (\|Q\|_{H^1} +\|v\|_{H^1})  \|v\|_{H^1}^2
\end{align*}
and, by \eqref{second_item_1} and \eqref{second_item_2},
\begin{align}
\|\nabla R(v)\|_{L^{\frac{6}{5}}} &\lesssim (\||\nabla|^{\frac{b+1}{3}} Q\|_{L^\frac{18}{7}}+\||\nabla|^{\frac{b+1}{3}} v\|_{L^\frac{18}{7}})\||\nabla|^{\frac{b+1}{3}} v\|_{L^\frac{18}{7}}^2\\
& \quad +(\|Q\|_{L^2} + \|v\|_{L^2})\|\nabla v\|_{L^2} \|v\|_{L^6}\\
&\quad +(\|\nabla Q\|_{L^2}\|v\|_{L^2} + \|Q\|_{L^2}\|\nabla v\|_2 + \|v\|\|\nabla v\|_{L^2}) \|v\|_{L^6}\\
&\lesssim (\|Q\|_{W^{1,\frac{18}{7}}}+\|v\|_{W^{1,\frac{18}{7}}})\|v\|_{W^{1,\frac{18}{7}}}^2 \\
& \quad + (\|Q\|_{H^1}+\|v\|_{H^1})\|v\|_{H^1}^2.
\end{align}
This yields
\begin{align}
    \|R(\mathcal{V}_1)\|_{W^{1,\frac{6}{5}}} \lesssim e^{-2e_0t}&\left[ (\|Q\|_{W^{1,\frac{18}{7}}}+e^{-e_0 t}\|Z_1\|_{W^{1,\frac{18}{7}}})\|Z_1\|_{W^{1,\frac{18}{7}}}^2
    \right.\\
    &\quad\left.+ (\|Q\|_{H^1}+e^{-e_0 t}\|Z_1\|_{H^1})\|Z_1\|_{H^1}^2\right],
\end{align}
which yields the base case.

Suppose now that $\mathcal{V}_1,\dots,\mathcal{V}_k$ are defined and define
\begin{equation}\label{epsilon-k}
    \epsilon_k = \partial_t \mathcal{V}_k + \mathcal{LV}_k - i R(\mathcal{V}_k).
\end{equation}
We then have 
\begin{equation}
    \partial_t \mathcal{V}_k = -\sum_{i=1}^k j e_0 e^{-je_0t}Z_k,
\end{equation}
which allows us to write
\begin{equation}\label{epsilon_expansion}
\epsilon_k = \sum_{j=1}^k e^{-je_0t}(-j e_0 Z_j+\mathcal{L}Z_j)-iR(\mathcal{V}_k).
\end{equation}
Using the explicit expression of $R(\mathcal{V}_j)$ and Lemma~\ref{spectral_decay}, we see that there exist $F_j \in W^{1,\frac{6}{5}}$ such that 
\begin{equation}
\epsilon_k(t,x) = \sum_{j=1}^{k+1} e^{-je_0t}F_j(x) + \mathcal{O}(e^{-e_0(k+2)t})
\end{equation}
for large $t$.  Since $\epsilon_k = \mathcal{O}(e^{-(k+1)e_0t})$ by the induction hypothesis, we conclude that $F_j = 0$ for $j \leq k$, showing
\begin{equation}
\epsilon_k(x,t) = e^{-(k+1)e_0t}F_{k+1}(x) + \mathcal{O}(e^{-e_0(k+2)t}).
\end{equation}

Noting that $(k+1)e_0\notin\sigma(\mathcal{L})$, we now define 
\[
Z_{k+1} = -(\mathcal{L}-(k+1)e_0)^{-1}F_{k+1} \in H^2.
\]
Note that, by Lemma~\ref{spectral_decay}, $Z_{k+1}$ also belongs to $W^{1,\frac{6}{5}}$. It therefore remains to estimate
\begin{equation}
    \epsilon_{k+1} = \epsilon_k - e^{-(k+1)e_0}F_{k+1}-i(R(\mathcal{V}_{k+1})-R(\mathcal{V}_k)).
\end{equation}
As we already know that $\epsilon_k - e^{-(k+1)e_0}F_{k+1} = \mathcal{O}(e^{-(k+2)e_0t})$, and the explicit expression of $R$ gives $R(\mathcal{V}_{k+1})-R(\mathcal{V}_k) = \mathcal{O}(e^{-(k+2)e_0t})$, we deduce the desired estimate.\end{proof}

Having constructed the approximate solutions, we now use a fixed point argument to obtain true solutions to \eqref{NLS}.

\begin{proposition}\label{sub_exist_UA} Let $A\in\R\backslash\{0\}$. There exists $k_0 > 0$ such that for any $k \geq k_0$, there exists $t_k \geq 0$ and a solution $U^A$ to \eqref{NLS} such that for $t \geq t_k$, we have

\begin{equation}\label{sub_bound_UA}
\|U^A-e^{it}Q-e^{it}\mathcal{V}_{k}^A\|_{Z([t,\infty))} \leq e^{-(k+\frac{1}{2})e_0 t},
\end{equation}
where the $\mathcal{V}_k^A$ are as in Proposition~\ref{family_appr}.

Furthermore, $U^A$ is the unique solution to \eqref{NLS} satisfying \eqref{sub_bound_UA} for large $t$. 

Finally, $U^A$ is independent of $k$ and satisfies
\begin{equation}\label{sub_bound_UA2}
	\|U^A(t) - e^{it}Q - Ae^{-e_0t+it}\mathcal{Y}_+\|_{H^1} \leq e^{-2e_0 t}\qtq{for large}t.
\end{equation}
\end{proposition}
\begin{proof}
We seek to construct a solution $u$ to \eqref{NLS} of the form
\begin{equation}\label{hktou}
u(t,x) = e^{it}(Q(x)+\mathcal{V}_k^A(t,x)+h_k^A(t,x)),
\end{equation}
which requires that we construct $h_k^A$ satisfying
\begin{equation}\label{linearized_vk}
i\partial_t h_k^A + \Delta h_k^A + K(h_k^A) + R(\mathcal{V}_k^A+h_k^A)-R(\mathcal{V}_k^A)+i\epsilon_k^A = 0,
\end{equation}
where $\epsilon_k$ is as in \eqref{epsilon-k}. For simplicity, throughout the proof, we write $h$ instead of $h^A_k$.  We construct $h$ by finding a fixed point of the following operator 
\begin{equation}
\mathcal{M}(h)(t) = i \int_t^{+\infty} e^{i(t-s)(\Delta-1)}\left\{K(h) + [R(\mathcal{V}_k^A+h)-R(\mathcal{V}_k^A)] + i\epsilon_k \right\} \, ds,
\end{equation}
which we will show is a contraction on a suitable complete metric space.  In particular, we define the norm $E=E(k,t_k)$ by
\[
\|w\|_E := \sup_{t \geq t_k} e^{(k+\frac{1}{2})t}\| h\|_{Z([t,\infty))}
\]
and take $B=B(k,t_k)$ to be the complete metric space 
\begin{align}
B &:= \{ h: \|h\|_E \leq 1\}
\end{align}
equipped with the metric
\begin{equation}
\rho(h,\tilde{h}) =     \|h - \tilde{h}\|_E.
\end{equation}

Then for $h, \tilde{h} \in B$, we have the following:
\begin{align}
\| \mathcal{M}(h)\|_{Z([t,\infty))}&\lesssim \| K(h)\|_{N([t,\infty))}+\|R(\mathcal{V}_k+h)-R(\mathcal{V}_k)\|_{N([t,\infty))}+\|\epsilon_k\|_{N([t,\infty))} 
\end{align}
and 
\begin{align}
\| \mathcal{M}(h)-\mathcal{M}(\tilde{h})\|_{Z([t,\infty))}&\lesssim \|K(h-\tilde{h})\|_{N([t,\infty))} \\
&\quad+\| R(\mathcal{V}_k+h)-R(\mathcal{V}_k+\tilde{h})\|_{N([t,\infty))}.
\end{align}

Now, by Lemma~\ref{prel_estimates}, for $0< \tau < 1$:
\begin{align}
\| K(h)\|_{N([t,\infty))}&\leq \sum_{j=0}^\infty  \| K(h)\|_{N([t+j\tau,t+(j+1)\tau])}\\
&\lesssim  \sum_{j=0}^\infty \tau^{\frac{1}{2}} \|  h\|_{Z([t+j\tau,t+(j+1)\tau])}\\
&\lesssim \tau^{\frac{1}{2}} \sum_{j=0}^\infty  e^{-(j+\frac{1}{2})e_0(t+j\tau)}\|h\|_E\\
&=\tau^{\frac{1}{2}} e^{-(k+\frac{1}{2})e_0t}  \|h\|_E\sum_{j=0}^\infty  e^{-j(k+\frac{1}{2})e_0\tau}\\
&= e^{-(k+\frac{1}{2})e_0t}\frac{\tau^{\frac{1}{2}}}{1-e^{-(k+\frac{1}{2})e_0\tau}}\|h\|_E.
\end{align}
Choosing $k_0 := \frac{\ln 2}{e_0} \frac{1}{\tau}-\frac{1}{2}$, we see that for all $k > k_0(\tau)$,
\begin{equation}
\|K(h)\|_{N([t,\infty))} \lesssim \tau^\frac{1}{2} e^{-(k+\frac{1}{2})e_0t}\|h\|_E.
\end{equation}

Similarly, for $I_j = [t+j\tau,t+(j+1)\tau]$, we have
\begin{align} 
&\| R(\mathcal{V}_k+h)-R(\mathcal{V}_k+\tilde{h})\|_{N(I_j)}\\ &\quad\lesssim 
\left[\| \mathcal{V}_k\|_{N(I_j)}+\| h\|_{N(I_j)}+ \| \tilde{h}\|_{Z(I_j)}]\right]\|h-\tilde{h}\|_{Z(I_j)}\\
&\quad\lesssim_k\left[e^{-e_0t} + e^{-(k+\frac{1}{2})e_0t}(\|h\|_E+\|\tilde{h}\|_E) \right]e^{-(k+\frac{1}{2})e_0(t+j\tau)}\|h-\tilde{h}\|_E\\
&\quad\lesssim_k e^{-(k+1)e_0t}\|h-\tilde{h}\|_E \, e^{-j(k+\frac{1}{2})\tau},
\end{align}
which yields
\begin{equation}
\| R(\mathcal{V}_k+h)-R(\mathcal{V}_k+\tilde{h})\|_{N([t,\infty))} \lesssim_k e^{-\left(k+1\right)e_0t}\|h-\tilde{h}\|_E.
\end{equation}

Finally, by construction,
\begin{equation}
\|\epsilon_k\|_{N([t,\infty))} \lesssim_k e^{-(k+1)e_0t}
\end{equation}
(cf. the proof of Proposition~\ref{family_appr}).

Collecting the estimates above, we see that for $t \geq t_k$, with $t_k>0$ large enough, we have
\begin{align*}
\|\mathcal{M}(h)\|_{E} &\leq \left[C \tau^{\frac{1}{2}} + C_k e^{-(k+\frac{1}{2})e_0 t_k}\right] \leq \tfrac{1}{2},\\
\|\mathcal{M}(h)-\mathcal{M}(\tilde{h})\|_{E} &\leq \left[C \tau^{\frac{1}{2}} + C_k e^{-(k+\frac{1}{2})e_0 t_k}\right]\|h-\tilde{h}\|_E \leq \tfrac{1}{2}\|h-\tilde{h}\|_E.
\end{align*}
Thus we obtain a unique fixed point $h_k^A$ for $\mathcal{M}$ in $B=B(k,t_k)$, which then yields the desired solution via \eqref{hktou}.   (To be clear, the parameters are chosen as follows: one first chooses a small universal $\tau>0$, then a large $k>k_0(\tau)$, and finally a large $t_k(k)$.) Note that the uniqueness condition still holds if, given $k>k_0$, one chooses a larger $\tilde{t}_k > t_k$. 

Finally, we show that the function $U_k^A := e^{it}(Q+\mathcal{V}_k + h_k)$ is independent of $k$.  Indeed, given $k'>k>k_0$, $t_{k'} > t_k$ and two solutions $h_k \in B(k,t_k)$ and $h_{k'} \in B(k',t_{k'})$, respectively, one obtains two solutions on $B(k,t_{k'})$, namely $h_k$ restricted to $t \in [t_{k'},\infty)$ and $\tilde{h}_k :=   (\mathcal{V}_{k'}-\mathcal{V}_k) + h_{k'}$. By uniqueness of \eqref{linearized_vk}, these must coincide on $[t_k',\infty)$, and hence by uniqueness of solutions to \eqref{NLS}, they must also coincide on $[t_k,\infty)$.

Finally, the bound \eqref{sub_bound_UA2} is obtained by writing $U^A = e^{it}(Q+Ae^{-e_0t}\mathcal{Y}_+)+\mathcal{O}(e^{-2e_0t})$ in $H^1(\mathbb{R}^3)$. \end{proof}

\section{Uniqueness for solutions converging exponentially to the ground state}\label{Sec:bootstrap}

In this section, we establish a uniqueness result for threshold solutions converging to the ground state.  The key technical ingredient will be the following proposition:
\begin{proposition}\label{prop_boot1}
Let $h$ be a solution to
\begin{equation}
\partial_th + \mathcal{L}h = \epsilon. 
\end{equation}
If there exist $c_2 > c_1>0$ such that for all $t>0$, 
\begin{align}\label{initial_rate}
\|h(t)\|_{H^1} &\lesssim e^{-c_1t},\\
\label{epsilon_rate}
    \|\epsilon\|_{N([t,\infty))} 
    &\lesssim e^{-c_2t},
\end{align}
then there exists $A\in \mathbb R$ such that
\begin{equation}\label{boostrapped_decay}
    \|h(t)-Ae^{-e_0t}\mathcal{Y}_+\|_{H^1} \lesssim e^{-\frac{c_1+c_2}{2}t}
\end{equation}
for all $t>0$.  Moreover, if $c_1>e_0$ or $c_2\leq e_0$, we can choose $A = 0$.
\end{proposition}

Let us first see how this proposition implies our desired results; we will then prove Proposition~\ref{prop_boot1} below. 



\begin{proposition}\label{uniqueness_generic_rate} Suppose $u:[0,\infty)\times\R^3\to\C$ is a solution to \eqref{NLS} such that
\begin{equation}\label{generic_decay_rate}
\|u(t)-e^{it}Q\|_{H^1} \lesssim e^{-ct}
\end{equation}
for some $c>0$.  Then there exists $A\in \mathbb R$ such that $u = U^A$ (cf. Proposition~\ref{sub_exist_UA}).
\end{proposition}

\begin{proof} For $A\in\R$, we define $h$ and $\mathcal{V}^A$ via
\[
u=e^{it}(Q+h) \qtq{and} U^A = e^{it}(Q+\mathcal{V}^A),
\]
where $U^A$ is as in Proposition~\ref{sub_exist_UA}.  Controlling the difference $u-U^A$ is then equivalent to controlling the difference $\tilde h:= h-\mathcal{V}^A$ for suitable $A$.  In particular, we will prove that there exists $A\in \R$ such that 
for all $k>1$ and all $t>0$, 
\begin{equation}\label{H1_difference_k}
\|\tilde h(t)\|_{H^1}=\|h(t)-\mathcal{V}^A(t)\|_{H^1}\lesssim_k e^{-kt}.    
\end{equation}
Assuming for \eqref{H1_difference_k} for the moment, let us derive a similar bound for $\tilde h$ in the $Z$-norm, which implies the desired conclusion $u = U^A$ from the uniqueness statement in Proposition~\ref{sub_exist_UA}. Noting that $h$ satisfies the equation \eqref{linearized_eq}, i.e.
\begin{equation}
\partial_t h + \mathcal{L}h = iR(h),
\end{equation}
we have that $\tilde h$ satisfies
\begin{equation}
i\partial_t \tilde h + \Delta \tilde h + K(\tilde h) + [R(\mathcal{V}^A+\tilde h)-R(\mathcal{V}^A)]= 0.
\end{equation}
Then, choosing $\tau\in(0,1)$ and $k>e_0$, we use the estimates \eqref{stein_lemma}, \eqref{first_item_1}, \eqref{first_item_2}, \eqref{first_item_3}, \eqref{diff_R}, \eqref{grad_diff_R}, and \eqref{H1_difference_k} to obtain
\begin{align}
\|\tilde h&\|_{Z([t,t+\tau])} \\
&\lesssim \|\tilde h(t)\|_{H^1_x} + \tau^{\frac{1}{2}}\|\tilde h\|_{Z([t,t+\tau])} + \left[\|\tilde h\|_{L^{\infty}_tH^1_x([t,t+\tau])}+\|\tilde h\|_{L^{\infty}_tH^1_x([t,t+\tau])}^2\right.\\
&\quad\left.+\|\mathcal{V}^A(t)\|_{L^{\infty}_tH^1_x([t,t+\tau])}+\|\mathcal{V}^A(t)\|_{L^{\infty}_tH^1_x([t,t+\tau])}^2\right]\||\nabla|^b \tilde h\|_{L^2_tL^6_x([t,t+\tau])}\\
&\lesssim \|\tilde h(t)\|_{H^1_x} + \tau^{\frac{1}{2}}\|\tilde h\|_{Z([t,t+\tau])} \\
&\quad+ \left(\|\tilde h\|_{L^{\infty}_tH^1_x([t,t+\tau])}+\|\tilde h\|_{L^{\infty}_tH^1_x([t,t+\tau])}^2+e^{-e_0t}\right)\|\tilde h(t)\|_{Z([t,t+\tau])}\\
&\label{H1_strichartz}
\lesssim_k e^{-kt}+ \tau^{\frac{1}{2}}\|\tilde h\|_{Z([t,t+\tau])} + e^{-kt}\|\tilde h\|_{Z([t,t+\tau])}.
\end{align}
This implies that for any $k>e_0$, there exists $0< \tau_k \ll 1$ and $t_k \gg 1$ such that for any $t>t_k$,
\begin{equation}\label{H1_strichartz_2}
\|\tilde h\|_{Z([t,t+\tau_k])} \lesssim_k e^{-kt}.
\end{equation}
Thus, splitting $[t,\infty) = \displaystyle\bigcup_{j=0}^{\infty} [t+j\tau_k,t+(j+1)\tau_k]$, we can write
\begin{equation}
\|\tilde h\|_{Z([t,\infty))} \lesssim_k e^{-kt}\sum_{j=0}^{\infty}e^{-j\tau_k} = \frac{1}{1-e^{-\tau_k}} e^{-kt},
\end{equation}
as desired.

We turn to the proof of \eqref{H1_difference_k}. Using \eqref{generic_decay_rate}, the inequalities \eqref{stein_lemma}, \eqref{diff_R} and \eqref{grad_diff_R}, and estimating as we did for \eqref{H1_strichartz} and, \eqref{H1_strichartz_2}, we have
\begin{equation}\label{hRh}
\|h(t)\|_{H^1}+\|h\|_{Z([t,\infty))}\lesssim e^{-ct} \qtq{and}\|R(h)\|_{N([t,\infty))} \lesssim e^{-2ct}. 
\end{equation}
We now claim that there exists $A\in\R$ such that 
\[
\|h(t)-\mathcal{V}^A(t)\|_{H^1}\lesssim e^{-\frac{3e_0}{2}t}. 
\]
Observing that (by construction of $U^A$)
\[
\mathcal{V}^A = Ae^{-e_0t}\mathcal{Y}_++ \mathcal{O}(e^{-2e_0t}),
\]
we see that it suffices to find $A_0\in \R$ such that
\begin{equation}\label{haets}
\|h(t)-A_0e^{-e_0t}\mathcal{Y}_+\|_{H^1} \lesssim e^{-\frac{3e_0}{2}t}.
\end{equation}
For this, we will utilize Proposition~\ref{prop_boot1}.  Using this proposition, the triangle inequality, an estimating as we did to obtain the bound $R(h)$ in \eqref{hRh}, we have the general implication
\[
\|h(t)\|_{H^1}\lesssim e^{-at} \implies \|h(t)\|_{H^1} \lesssim e^{-\min\{e_0,\frac32a\}t}.
\]
In particular, starting with \eqref{hRh}, after finitely many iterations (say $J$, where $J$ is large enough that $(\tfrac32)^J c>e_0$), we obtain the decay estimate $\|h(t)\|_{H^1}\lesssim e^{-e_0 t}$.  From this point, one more application of Proposition~\ref{prop_boot1} (and \eqref{hRh}) implies \eqref{haets} for suitable $A_0\in\R$.  

Now, recall that $\tilde h= h-\mathcal{V}^{A_0}$ satisfies
\begin{equation}
\partial_t \tilde h + \mathcal L \tilde h = i[R(\tilde h+\mathcal{V}^{A_0})-R(\mathcal{V}^{A_0})],
\end{equation}
so that we can derive
\[
\|\tilde h(t)\|_{H^1}\lesssim e^{-mt} \implies \|R(\tilde h+\mathcal{V}^{A_0})-R(\mathcal{V}^{A_0})\|_{N([t,\infty)}\lesssim e^{-(2e_0+m)t}.
\]
Thus, beginning with $m=\tfrac32e_0$, repeated applications of Proposition~\ref{prop_boot1} (in which we can always take `$A=0$') yield
\[
\|\tilde h(t)\|_{H^1} \lesssim e^{-mt} \implies \|\tilde h(t)\|_{H^1}\lesssim e^{-(m+e_0)t}.
\]
This implies \eqref{H1_difference_k} (with $A=A_0$) and completes the proof.\end{proof}


The same argument as in the previous proof also shows:
\begin{corollary}\label{C:exponential_classification_1}
If $A \in \mathbb R$ and $u$ is a solution to \eqref{NLS} on $[t_0,\infty)$ such that, for all large $t$, 
\begin{equation}
\|u(t)-U^A(t)\|_{H^1} \lesssim e^{-ct},    
\end{equation}
where $c>e_0$, then $u = U^A$.
\end{corollary}

We can therefore reduce even further the number of possible special solutions (up to phase and time translations).
\begin{corollary}\label{C:exponential_classification_2}
For any $A >0$, there exists $T_A$ such that $U^A = e^{iT_A}U^{+1}(\cdot+T_A)$. Similarly, for any $A <0$, there exists $T_A$ such that $U^A = e^{iT_A}U^{-1}(\cdot+T_A)$.
\end{corollary}
\begin{proof}
If $A>0$, let $T_A = \frac{1}{e_0}\ln A$. Then
\begin{align}
\|e^{-iT_A}&U^A(t)-U^{+1}(t+T_A)\|_{H^1} \\
&= \|e^{-iT_A}U^A(t)-e^{i(t+T_A)}(Q+e^{-e_0(t+T^A)}\mathcal{Y}_+)\|_{H^1}+\mathcal{O}(e^{-2e_0t}) \\
&= \|U^A(t)-e^{it}(Q+Ae^{-e_0t}\mathcal{Y}_+)\|_{H^1} +\mathcal{O}(e^{-2e_0t})  \\
&=\mathcal{O}(e^{-2e_0t}).
\end{align}
Therefore, by the previous corollary, $U^A = e^{iT_A}U^{+1}(\cdot+T_A)$. The proof is analogous for the case $A<0$.
\end{proof}

Finally, we need to prove Proposition~\ref{prop_boot1}. 

\begin{proof}[Proof of Proposition~\ref{prop_boot1}] Throughout the proof, we employ the notation introduced in Section~\ref{Sec:spectral}.  We first observe that \eqref{initial_rate} and \eqref{epsilon_rate}, together with Strichartz, \eqref{first_item_1}. \eqref{first_item_2} and \eqref{first_item_3}, imply 
\[
\|h\|_{Z([t,\infty))} \lesssim e^{-c_1t}.
\]

Recall that 
\[
|B(\mathcal{Y}_+,\mathcal{Y}_-)| = |(L_-\mathcal{Y}_1,\mathcal{Y}_1)| \gtrsim \|\mathcal{Y}_1\|_{H^1} > 0.
\]

We now renormalize $\mathcal{Y}_\pm$ as to have $B(\mathcal{Y}_+,\mathcal{Y}_-)=1$. We decompose 
\[
h(t) = \alpha_+(t) \mathcal{Y}_++ \alpha_-(t) \mathcal{Y}_- + \beta(t) \frac{iQ}{\|Q\|_2} + g(t),
\]
with $g\in \tilde{G}^{\perp}$.  In particular,
\begin{align}
\alpha_\pm=B(h,\mathcal{Y}_\pm), \quad 
    \beta = \frac{(h,iQ)}{\|Q\|_2},\qtq{and}    \Phi(h) = \Phi(g) + \alpha_+ \alpha_-. 
\end{align}

Moreover, we have
\[
|\alpha_+(t)|+|\alpha_-(t)| +|\beta(t)| + \|g(t)\|_{H^1} \lesssim e^{-c_1t}.
\]

By differentiation, we then have 
\begin{align}
&\tfrac{d}{dt}(e^{e_0t}\alpha_+) = e^{e_0t}B(\epsilon,\mathcal{Y}_-)\label{ode_plus}\\
&\tfrac{d}{dt}(e^{-e_0t}\alpha_-) = e^{-e_0t}B(\epsilon,\mathcal{Y}_+)\label{ode_minus}\\ &\tfrac{d}{dt}\beta = \tfrac{(\epsilon,iQ)}{\|Q\|_2}\label{ode_beta}\\
 &\tfrac{d}{dt}\Phi(h) = 2B(\epsilon,h)\label{ode_h}.
\end{align}

We first claim that
\begin{equation}\label{boot1}
    |\alpha_-(t)| +|\beta(t)| + \|g(t)\|_{H^1} \lesssim e^{-\frac{c_1+c_2}{2} t}.
\end{equation}
Indeed, using H\"older and \eqref{second_item_1}, we first observe the general estimate
\begin{align}
\int_I |B(f(t),g(t))| \, dt & \lesssim (1+|I|^{\frac{1}{2}})\| f\|_{N(I)}\| g\|_{Z(I)}.
\end{align}
Then, writing $[t,+\infty) = \bigcup_{i=0}^{\infty}[t+i,t+i+1)$ and integrating \eqref{ode_minus}, \eqref{ode_beta} and \eqref{ode_h} over $t\in[0,\infty)$, we obtain \eqref{boot1}.

We now consider $\alpha_+$. We split the remaining of the proof in three cases:

\textbf{Case 1: $c_1> e_0$.} As $e^{e_0t}\alpha_+(t) \to 0$ as $t \to \infty$, we can integrate \eqref{ode_plus} over $t\in[0,\infty)$ to obtain
\begin{equation}
    |\alpha_+(t)|\lesssim e^{-2c_2t} \leq e^{-\frac{c_1+c_2}{2} t}.
\end{equation}
Combining this estimate with \eqref{boot1}, we derive \eqref{boostrapped_decay} with $A=0$.

\textbf{Case 2: $c_1\leq e_0 < c_2$.}  As $c_2>e_0$, we can still integrate \eqref{ode_plus} over $t\in[0,\infty)$ to obtain $A\in \mathbb R$ such that $e^{e_0t}\alpha_+(t) \to A$ as $t \to \infty$. Moreover,
\begin{equation}
    |\alpha_+(t)-e^{-e_0t}A| \lesssim e^{-c_2 t}\leq e^{-\frac{c_1+c_2}{2} t}.
\end{equation}
 This, together with \eqref{boot1}, yields \eqref{boostrapped_decay}.

\textbf{Case 3: $c_2 \leq e_0$.} Integrating \eqref{ode_plus} over $[t_0,t]$ (for some $t_0\in\R$), we have
\begin{equation}\label{boot2}
    |\alpha_+(t)| \lesssim e^{-e_0 t} \left[e^{e_0t_0}|\alpha_-(t_0)|+ \int_{t_0}^te^{(e_0-c_2) \tau}\, d\tau\right] \lesssim e^{-\frac{c_1+c_2}{2} t}.
\end{equation}
In particular, using \eqref{boot1}, we obtain \eqref{boostrapped_decay}, with $A=0$.\end{proof}

\section{Proof of the main results}\label{Sec:closure}

Finally, we collect the results from the preceding sections to complete the proofs of the main results, namely, Theorem~\ref{T:Existence_special_solutions} and Theorem~\ref{T:Classification_threshold_solutions}.

First, we describe the particular solutions $Q^\pm$. 
\begin{proof}[Proof of Theorem~\ref{T:Existence_special_solutions}]
Define $Q^\pm = U^{\pm1}$ (cf. Proposition~\ref{sub_exist_UA}). Since
\begin{equation}
\|\nabla Q^{\pm}\|_{L^2}^2 = \|\nabla Q\|_{L^2}^2  \pm 2e^{-e_0t} \|\nabla \mathcal{Y}_+\|_{L^2}^2+\mathcal{O}(e^{-2e_0t}),
\end{equation}
we see that $\|\nabla Q^+(t)\|_{L^2}>\|\nabla Q\|_{L^2}$ and $\|\nabla Q^-(t)\|_{L^2}<\|\nabla Q\|_{L^2}$ for all large $t$.

Now, if $Q^{-}$ does not scatter backwards in time, then $\{Q^{-}(t) \colon t\in \mathbb{R} \}$ is pre-compact in $H^1$. By time-reversal, Lemma~\ref{L:virial}, Corollary~\ref{C:mod-virial}, and Corollary~\ref{const_modul}, we have 
\begin{equation}
\int_\R \delta(Q^{-}(t)) \, dt \lesssim \lim_{t\to\infty}\delta(Q^{-}(t))+\lim_{t\to-\infty}\delta(Q^{-}(t)) = 0,
\end{equation}
so that $\delta(Q^-(t))\equiv 0$.  However, this contradicts that $\|\nabla Q^-(t)\|_{L^2}<\|\nabla Q\|_{L^2}$.

Next, since $Q^+$ is radial, Proposition~\ref{P:global-unconstrained_finite_variance} implies $|x|Q^{+}\in L^2$; moreover,
\begin{equation}
     2\Im\int Q^{+}_0 \, [x\cdot\nabla\overline{Q^+_0}]  >0.
\end{equation}
If $Q^{+}$ does not blow up in finite negative time, then applying the same Proposition~\ref{P:global-unconstrained} to the time-reversed solution $\overline{Q^{+}}(t) = {Q^+}(-t)$, we derive the contradiction 
\begin{equation}
    2\Im\int \overline{Q^{+}_0} \, [x\cdot\nabla Q^+_0]  >0.
\end{equation}
\end{proof}

Finally, we have the classification result. 
\begin{proof}[Proof of Theorem~\ref{T:Classification_threshold_solutions}] Let $u_0$ be such that $M(u_0)^{1-s_c}E(u_0)^{s_c} = M(Q)^{1-s_c}E(Q)^{s_c}
$. Using the scaling symmetry, we may assume $M(u_0)=M(Q)$ and $E(u_0)=E(Q).$

If $\|\nabla u_0\|_{L^2}<\|\nabla Q\|_{L^2}$ and the corresponding solution $u$ does not scatter in both time directions, then Propositions~\ref{P:compactness} and \ref{P:constrained-converge}, Corollaries~\ref{C:exponential_classification_1} and \ref{C:exponential_classification_2}, and Theorem~\ref{T:Existence_special_solutions} imply that $u = Q^-$ (up to the symmetries of the equation). 

Similarly, if $\|\nabla u_0\|_{L^2}>\|\nabla Q\|_{L^2}$ and $u$ does not blow up in both time directions, then by Proposition~\ref{P:global-unconstrained}, Corollaries~\ref{C:exponential_classification_1} ,and \ref{C:exponential_classification_1} and Theorem~\ref{T:Existence_special_solutions}, we have $u=Q^+$ (up to the symmetries of the equation).\end{proof}







\bibliographystyle{abbrvnat}
\bibliography{biblio}

\newcommand{\Addresses}{{
  \bigskip
  \footnotesize

  L. Campos, \textsc{IMECC, State University of Campinas (UNICAMP), Campinas, SP, Brazil}\par\nopagebreak
  \textit{E-mail address:} \texttt{luccas@ime.unicamp.br}

  \medskip

  J. Murphy, \textsc{Department of Mathematics \& Statistics, Missouri University of Science \& Technology, Rolla, MO, USA}\par\nopagebreak
  \textit{E-mail address:} \texttt{jason.murphy@mst.edu}

}}
\setlength{\parskip}{0pt}
{
\Addresses}

\end{document}